\documentclass[10pt,reqno]{article} 
\usepackage{amsfonts,amsmath,layout}
\usepackage{mathrsfs}  
\usepackage{soul}
\usepackage{amssymb,mathabx,amsfonts,amsthm,amscd,stmaryrd,dsfont,esint,upgreek,constants,todonotes}
\usepackage{enumitem}

\DeclareFontFamily{OT1}{pzc}{}
\DeclareFontShape{OT1}{pzc}{m}{it}{<-> s * [1.10] pzcmi7t}{}
\DeclareMathAlphabet{\mathpzc}{OT1}{pzc}{m}{it}

\usepackage{color} 

\definecolor{Red}{cmyk}{0,1,1,0.2}

\newcommand{\R}{\mathbb R}

\def\R{\mathbb R}

\def\a{\alpha}
\def\e{\varepsilon}

\newcommand{\be}{\begin{equation}}
\newcommand{\ee}{\end{equation}}
\def\1{{\bf 1}}

\newtheorem{Theorem}{Theorem}[section]

\newtheorem{defi}[Theorem]{Definition}
\newtheorem{theo}[Theorem]{Theorem}

\newtheorem{Proposition}[Theorem]{Proposition}
\newtheorem{Lemma}[Theorem]{Lemma}
\newtheorem{Corollary}[Theorem]{Corollary}
\newtheorem{Remark}[Theorem]{Remark}

\begin{document}

\title{\bf Beyond uniqueness:\\
Relaxation calculus of junction conditions \\for coercive Hamilton-Jacobi equations}
\author{\renewcommand{\thefootnote}{\arabic{footnote}}
  N. Forcadel\footnotemark[1], R. Monneau\footnotemark[2] ~\footnotemark[3]}
\footnotetext[1]{INSA Rouen Normandie, Normandie Univ, LMI UR 3226, F-76000 Rouen, France.}
\footnotetext[2]{CEREMADE, UMR CNRS 7534, Universit\'e Paris Dauphine-PSL,
Place de Lattre de Tassigny, 75775 Paris Cedex 16, France. }
\footnotetext[3]{CERMICS, Universit\'e Paris-Est, Ecole des Ponts ParisTech, 6-8 avenue Blaise Pascal, 77455 Marne-la-Vall\'ee Cedex 2, France.}

\maketitle

\begin{abstract}
A junction is a particular network given by the collection of $N\ge 1$ half lines $[0,+\infty)$ glued together at the origin. On such a junction, we consider evolutive Hamilton-Jacobi equations with $N$ coercive Hamiltonians. Furthermore,
we consider a general desired junction condition at the origin, given by some monotone function $F_0:\R^N\to \R$.
There is existence and uniqueness of solutions which only satisfy weakly the junction condition (at the origin, they satisfy either the desired junction condition or the PDE).

We show that those solutions satisfy strongly a relaxed junction condition $\frak R F_0$ (that we can recognize as an effective junction condition). It is remarkable that this relaxed condition can be computed in three different but equivalent  ways: 1) using viscosity inequalities, 2) using Godunov fluxes, 3) using Riemann problems.
Our result goes beyond uniqueness theory, in the following sense:  solutions to two different desired junction conditions $F_0$ and $F_1$
do coincide if $\frak R F_0=\frak R F_1$.
\end{abstract}

\tableofcontents

\section{Introduction} 
In this paper, we consider Hamilton-Jacobi equations of evolution type posed on junctions. The first results concerning these equations have been obtained in the convex case in  \cite{ACCT, IMZ} and are closely related to optimal control. In \cite{IM1}, the authors obtained a strong comparison principle using a PDE approach while the non-convex case were studied in \cite{LS1, LS2} (see also \cite{FIM2, FIM3} for a new approach for proving comparison principle). Many contributions followed these  articles and the reader is referred to the book by G. Barles and E. Chasseigne \cite{BC-book} for an up-to-date state of the art.

 It is now well-known that if an equation is posed on a domain, the boundary condition can be in conflict with the equation and the same phenomena naturally appears when considering equations on networks. A classical way to handle this difficulty for Hamilton-Jacobi equations is to impose either the boundary condition (or the junction condition) or the equation at the boundary (or at the junction point), both in the viscosity sense. Such solutions are called weak (viscosity) solutions.

Concerning Hamilton-Jacobi on junctions and for convex and coercive Hamiltonians, C. Imbert and the second author \cite{IM1} show the existence of weak solutions and proved that these weak solutions satisfy other junction condition in a strong way. These conditions are called relaxed junction conditions and are parametrized by a single real parameter, called the flux limiter (only for convex Hamiltonians).

When the Hamiltonians is still coercive but not necessarily convex, the situation is much more complicated. The first result in that direction was obtained by J. Guerand \cite{guerand} in the one-dimensional case (a very simple junction with only one branch). She proved that, in this situation, it is still possible to relax the boundary conditions to obtain strong solutions. She showed in particular that the family of relaxed boundary conditions is much more rich and is characterized by a family of limiter points. In \cite{FIM1}, with C. Imbert, we revisit this characterization and propose a new formula which can be easily derived from the definition of weak viscosity solutions. We also exhibit a strong connection between this relaxed boundary condition and  Godunov's fluxes for conservation laws. 

The goal of this paper is to extend these results to the case of general junctions with several branches. In particular we show how to define properly relaxed junction conditions using the definition of weak viscosity solutions. As for the case of one branch, we show that the relaxation can also be defined in terms of Godunov's fluxes (see also \cite{M} for related results). Finally, we also give a third formulation of the relaxation based on Riemann problems.

Let us emphasise the fact that characterization of relaxed junction conditions is important, because several junction conditions can lead to the same relaxed junction condition.
Hence, even if at first glance, problems may seem different, they can be exactly the same (see Subsection \ref{s5.5}).
From this point of view, our analysis goes beyond the classical question of uniqueness of solutions.

\section{Main results}

\subsection{The junction problem}

We begin to  describe what is a (one-dimensional) {\bf junction}.
We consider $N\ge 1$ copies of the interval $[0,+\infty)$ that we call $J_\alpha$ for $\alpha=1,\dots,N$.
We glue all the branches $J_\alpha$ together at the origin such that
$$J_\alpha\cap J_\beta = \left\{0\right\}\quad \mbox{for all}\quad \alpha\not= \beta$$
We call the junction
$$J=\bigcup_{\alpha=1,\dots,N} J_\alpha.$$
where $x=0$ is now the {\bf junction point}.

For a function $u: [0,+\infty)\times J \to \R$ whose values are  $u(t,x)$, we denote by $u_t$ the time derivative of $u$, and
we define the gradient of $u$ as
$$u_x(t,x)=\left\{\begin{array}{ll}
\partial_\alpha u(t,x) & \quad \mbox{if}\quad x\in J_\alpha^*:=J_\alpha\backslash \left\{0\right\},\\
(\partial_1u(t,0^+), \dots, \partial_Nu(t,0^+)) & \quad \mbox{if}\quad x=0.\
\end{array}\right.$$
Throughout the paper, we will consider solutions of the following problem
\begin{equation}\label{eq::m10}
\left\{\begin{array}{ll}
{u}_t + {H}^\alpha({u}_x)=0 &\quad \mbox{for all} \quad t\in (0,+\infty) \quad \mbox{and}\quad x\in J_\alpha^*,\quad \mbox{for}\quad \alpha=1,\dots,N\\
{u}_t + {F_0}(u_x)=0 &\quad \mbox{for all} \quad t\in (0,+\infty) \quad \mbox{and}\quad x=0.\\
\end{array}\right.
\end{equation}
The second line of  \eqref{eq::m10} is the junction condition (JC) and by abuse of terminology, we will also say that the junction condition
(or junction function) is $F_0$.
In the particular case of  \eqref{eq::m10}, we call the second line, a {\bf desired junction condition}
(because as we will see, this condition can not always be truly satisfied).
Here the Hamiltonians ${H}^\alpha$ for $\alpha=1,\dots,N$ are assumed to satisfy
\begin{equation}\label{eq::p00}
\left\{\begin{array}{ll}
\mbox{\bf (Continuity)} & \quad {H^\alpha}\in C(\R^N)\\
\\
\mbox{\bf (Coercivity)} & \quad {H^\alpha}(p^\alpha)\to +\infty \quad \mbox{as}\quad \left|p^\alpha\right|\to +\infty.
\end{array}\right.
\end{equation}
When this is the case, also simply say  that 
$$H=(H^1,\dots,H^N)$$
satisfies \eqref{eq::p00}.
We also assume that the function $F_0$ satisfies
\begin{equation}\label{eq::m11}
\left\{\begin{array}{ll}
\mbox{\bf (Continuity)} & \quad {F_0}\in C(\R^N)\\
\\
\mbox{\bf (Monotonicity)} & \quad {F_0}\textrm{  is nonincreasing in each of its arguments.}
\end{array}\right.
\end{equation}
For existence results, we will also need an initial condition
\begin{equation}\label{eq::p21}
u(0,x)=u_0(x) \quad \mbox{for all}\quad x\in J
\end{equation}
that is assumed to be chosen such that
\begin{equation}\label{eq::p22}
\mbox{$u_0$ is uniformly continuous on $J$.}
\end{equation}

\subsection{The relaxation formula}
Inspired by \cite{FIM1}, in order to define the relaxation operator, we define  the sub-relaxation operator with respect to $H$ by
$$\underline{R} F_0(p) =\sup_{q\ge p} \min\left\{F_0(q),H_{min}(q)\right\}$$
and the super-relaxation operator by
$$\overline{R} F_0(p) =\inf_{q\le p} \max\left\{F_0(q),H_{max}(q)\right\}$$
where $q\le p$ means $q^\alpha\le p^\alpha$ for each $\alpha=1,\dots,N$ and 
$$H_{max}(p)=\max_{\alpha=1,\dots,N}H^\alpha(p^\alpha)\quad\quad  \mbox{and}\quad \quad
H_{min}(p)=\min_{\alpha=1,\dots,N}H^\alpha(p^\alpha).$$

We then have the following theorem which enables to define the relaxation operator $\frak R$.

\begin{Theorem}[The relaxation operator $\frak R$]\label{th:relax}
Let $H$ satisfy \eqref{eq::p00} and  $F_0$  satisfy \eqref{eq::m11}. Then
$$\underline R \overline R F_0=\overline R \underline R F_0\quad (=:\frak RF_0).$$
Moreover
$$\frak R F_0=\frak R \max(F_0,H_-)\ge H_-,$$
where
\begin{equation}\label{eq::rep1}
H_-(p)=\max_{\alpha=1,\dots,N}H^\alpha_-(p^\alpha) \quad \mbox{with}\quad H^\alpha_-(p^\alpha)=\inf_{q^\alpha\le p^\alpha} H^\alpha(q^\alpha)
\end{equation}
i.e. $H^\alpha_-$ is the  lower nonincreasing hull of $H^\alpha$.
\end{Theorem}

In Section \ref{sec:3}, we will introduce another relaxation formula using the Godunov's fluxes while in Section \ref{sec:7}, we introduce a third equivalent formulation based on Riemann problems.

\subsection{Weak and strong viscosity solutions}
As mentioned in the introduction, we can consider two types of solutions for Hamilton-Jacobi equations on junctions. The first one, called weak viscosity solution requires that either the junction or the equation is satisfied at the junction point (see Definition \ref{defi::1}), while the second one, called strong viscosity solutions requires that the junction condition is always satisfied (see Definition \ref{defi::2}). 

The second main result of the paper is to show that a function is a weak viscosity solution of \eqref{eq::m10} if and only if it is a strong viscosity solution of the same equation but with $F_0$ replaced by  the relaxation of $F_0$, namely $\frak R F_0$. We refer to Proposition \ref{pro:wtostr} for the precise result.\bigskip

\paragraph{Organization of the paper}
In Section \ref{sec:3}, we introduce the Godunov relaxation and prove some properties that will be useful to study the relation of this relaxation with the relaxation operator $\frak R$. In Section \ref{sec:4} we show that the Godunov relaxation and the relaxation operator $\frak R$ coincide and we prove further properties on the relaxation operator.
Section \ref{sec:5} is devoted to the study of viscosity solutions for \eqref{eq::m10}. We show in particular that a function is a $F_0$-weak solution iff it is a $\frak R F_0$-strong solution to the same equation. In Section \ref{sec:6}, we state a comparison principle for \eqref{eq::m10}. Finally, in Section \ref{sec:7}, we propose a third relaxation formula based on Riemann problems and we show that it coincides with the relaxation operator $\frak R$.
\section{Godunov relaxation}\label{sec:3}
In this section, we introduce the Godunov relaxation. The precise definition is given in Subsection \ref{subsec:3.1}, while in Subsection \ref{subsec:3.2}, we introduce the semi-Godunov relaxation and present some important properties that will be useful in the sequel, in particular in Section \ref{sec:4} for the study of the relation between Godunov relaxation and the relaxation operator $\frak R$. 

Il all this section, we will assume that the junction condition satisfies the semi-coercivity assumption:

\begin{equation}\label{eq::p63}
F_0(p)\to +\infty \quad \mbox{as}\quad \min_{\alpha=1,\dots,N} p^\alpha \to -\infty .
\end{equation}

\subsection{Definition and properties of the Godunov relaxation}\label{subsec:3.1}

Recall that the Godunov flux associated to the Hamiltonian (or flux) $H^\alpha$ is given for $p,q\in\R$ by
\begin{equation}\label{eq::reg43bis}
G^\alpha(p,q)=\left\{\begin{array}{ll}
\displaystyle \min_{[p,q]}H^\alpha & \quad \mbox{if}\quad p\le q,\\
\displaystyle \max_{[q,p]}H^\alpha & \quad \mbox{if}\quad p\ge q.
\end{array}\right.
\end{equation}
In particular, $G^\alpha$ is non-decreasing in the first variable and non-increasing in the second one. Moreover, we have $G^\alpha(p,p)=H^\alpha(p)$.
We  define next the action of the Godunov flux on a semi-coercive, continuous and non-increasing function $F_0$.

\begin{Proposition}[Godunov's relaxation]\label{pro:n1}
Let $H$ satisfy \eqref{eq::p00} and  consider $F_0$  satisfying \eqref{eq::m11} and the semi-coercivity property \eqref{eq::p63}.
 Let $p\in\R^N$, then the following properties hold true.
 \begin{enumerate}[label=(\roman*)]
 \item \label{lambdaq}
   There exists at least one $q \in \R^N$ such that $F_0(q)=G^\alpha(q^\a,p^\a)$ for all $\alpha=1,\dots,N$. The common value is denoted by  $\lambda_q$.
 \item \label{lambda}
   The value $\lambda_q$ defined above is independent on $q$. We denote this unique value by 
   $$\lambda=\lambda(p)=:(F_0G)(p).$$
 \end{enumerate}
\end{Proposition}

\begin{proof}
In all the proof, $p\in \R^N$ is fixed.

We begin to prove \ref{lambdaq}. In order to get an increasing function, we define for $\varepsilon \ge 0$
$$K^\a_\e(q^\a)=G^\a(q^\a,p^\a)+\e(1+\tanh q^\a).$$
For  $\varepsilon>0$ and $\displaystyle{\lambda> \lambda_\e :=\max_{\a\in\{1,\dots,N\}}\inf K^\a_\e}\ge \max_{\alpha=1,\dots,N} \inf K^\alpha_0$, we define $q^\a_{\lambda,\e}$ as the unique value such that $K^\a_\e(q^\a_{\lambda,\e})=\lambda$. Since $K^\a_\e$ is increasing, we deduce that $\lambda\mapsto q^\a_{\lambda,\e}$ is increasing and continuous. We then define the continuous and increasing function, for $\lambda>\lambda_\e$
$$N_\e(\lambda)=\lambda-F_0(q_{\lambda,\e}).$$
Since $\min_{\a} (q_{\lambda,\e}^\a)\to -\infty$ as $\lambda\to\lambda_\e$, we deduce, by semi-coercivity of $F_0$ that $\lim_{\lambda\to\lambda_\e}N_\e(\lambda)=-\infty$. In the same way, since $q_{\lambda,\e}^\a\to +\infty$ as $\lambda\to+\infty$ for all $\a\in\{1,\dots,N\}$, we deduce that
$\lim_{\lambda\to+\infty}N_\e(\lambda)\ge \lim_{\lambda\to+\infty}\lambda -F_0(0)=+\infty$. We can then define
 $\lambda_\e^*$ as the unique value such that $N_\e(\lambda_\e^*)=0$. In particular, since $\e \mapsto q^\a_{\lambda,\e}$ is decreasing, we deduce that $\e\mapsto N_\e(\lambda)$ is non-increasing and so $\e\mapsto \lambda^*_\e$ is non-decreasing. This implies that $\lambda^*_\e$ is uniformly bounded for $\e\le 1$, by a constant depending only on $p$.   We then set
$$q_\e=q_{\lambda_\e,\e},$$
which satisfies
\begin{equation}\label{eq:01}
\lambda_\e^*=F_0(q_\e)=K_\e^\a(q^\a_\e)=G^\a(q^\a_\e,p^\a)+\e(1+\tanh q^\a_\e)\quad \textrm{ for all }\a=1,\dots, N.
\end{equation}
By semi-coercivity of $F_0$ and by the fact that $G(q^\a,p^\a)\to +\infty$ as $q^\a\to+\infty$, we deduce, using also that $\lambda^*_\e$ is uniformly bounded, that $q_\e$ is uniformly bounded. Up to extract a subsequence, we can then assume that $(\lambda^*_\e,q_\e)\to(\lambda^*,q)$ as $\e\to0$.  Passing to the limit in \eqref{eq:01}, we get, by continuity, that
$$\lambda^*=F_0(q)=G^\a(q^\a,p^\a)\quad \textrm{ for all }\a=1,\dots, N,
$$
which proves \ref{lambdaq}.

We now turn to \ref{lambda}. Assume that there exists $q_1, q_2$ such that $\lambda_{q_1}<\lambda_{q_2}$. We then have for all $\a\in \{1,\dots,N\}$
$$F_0(q_1)=G^\a(q^\a_1,p^\a)=\lambda_{q_1}<\lambda_{q_2}=F_0(q_2)=G^\a(q^\a_2,p^\a).$$
By monotonicity of the Godunov fluxes, we deduce that $q_1^\a<q_2^\a$ for $\a\in \{1,\dots,N\}$. Using the monotonicity of $F_0$, this implies that
$$\lambda_{q_1}=F_0(q_1)\ge F_0(q_2)=\lambda_{q_2}$$
which is a contradiction. This ends the proof of the proposition.
\end{proof}

%
%
%
%

In order to study the relation between the relaxation operator $\frak R$ and the Godunov relaxation, we need to introduce the Godunov semi-fluxes.
\subsection{Godunov semi-fluxes}\label{subsec:3.2}

For $\a=1,\dots, N$, we introduce the Godunov semi-fluxes, $\underline G^\a$ and $\overline G^\a$, which are set-valued applications defined by
\[ \underline G^\a(q^\a,p^\a)=
  \begin{cases}
    \{-\infty\}& \text{ if } q^\a<p^\a \medskip\\
    {[-\infty, H^\a(p^\a)]}&\text{ if }q^\a=p^\a, \medskip\\
   \displaystyle \left\{\max_{[p^\a,q^\a] }H^\a \right\}&\text{ if } q^\a>p^\a
  \end{cases}
\]
and
\[
  \overline G^\a(q^\a,p^\a)=
  \begin{cases}
    \displaystyle \left\{\min_{[q^\a,p^\a] }H^\a \right\}&\text{ if }q^\a<p^\a, \medskip\\
    {[ H^\a(p), +\infty]}&\text{ if }q^\a=p^\a, \medskip\\
    \{+\infty\} &\text{ if } q^\a>p^\a.
  \end{cases}
\]
As before, we can define the action of these semi-fluxes on non-increasing semi-coercive continuous functions. 
\begin{Proposition}[Lower Godunov's relaxation]\label{pro:n2}
Let $H$ satisfy \eqref{eq::p00} and  $F_0$  satisfy \eqref{eq::m11} and the semi-coercivity property \eqref{eq::p63}.
 Let $p\in\R^N$, then the following properties hold true.
 \begin{enumerate}[label=(\roman*)]
 \item \label{lambdaq2}
   There exists at least one $q \in \R^N$ such that $F_0(q)\in \underline G^\alpha(q^\a,p^\a)$ for all $\alpha=1,\dots,N$. The value $F_0(q)$ is denoted by  $\lambda_q$. Moreover, all the $q$ satisfying the previous property satisfy $q\ge p$.
 \item \label{lambda2}
   The value $\lambda_q$ defined above is independent on $q$. We denote this unique value by 
   $$\lambda=\lambda(p)=:(F_0\underline G)(p).$$
 \end{enumerate}
\end{Proposition}

\begin{proof}
Let $p\in \R^N$. If $\displaystyle{F_0(p)\le \min_{\a=1,\dots,N}H^\a(p^\a)}$, then we can choose $q^\a=p^\a$ for all $\a=1,\dots,N$ and we get $F_0(q)=F_0(p)\in \underline G^\a(q^\a,p^\a)=[-\infty,H^\a(p^\a)]$. 
If $\displaystyle{F_0(p)> \min_{\a=1,\dots,N}H^\a(p^\a)}$, we denote by $I^-:=\{\a, \; F_0(p)\le H^\a(p^\a)\}$ and by $I^+:=\{\a, \; F_0(p)> H^\a(p^\a)\}$. For $\a\in I^-$, we set $q^\a=p^\a$. In particular, $F_0(p)\in \underline G^\a(q^\a,p^\a)$ for $\a\in I^-$. Arguing as in the proof of Proposition \ref{pro:n1}, but working only on the indices $\alpha\in I^+$, we can construct $q^\a\ge p^\a$ such that $F_0(q)=G^\a(q^\a,p^\a)=\displaystyle \max_{[p^\alpha,q^\alpha]}H^\alpha$  for all $\a\in I^+$. Since we also have $F_0(q)\le F_0(p)\le H^\a(p^\a)$ for all $\a\in I^-$, we get \ref{lambdaq2}.
\medskip

The proof of \ref{lambda2} is similar to the one of Proposition \ref{pro:n1} \ref{lambda}.

\end{proof}

In the same way, we have the following result

\begin{Proposition}[Upper Godunov's relaxation]\label{pro:n3}
Let $H$ satisfy \eqref{eq::p00} and  $F_0$  satisfy \eqref{eq::m11} and the semi-coercivity property \eqref{eq::p63}.
 Let $p\in\R^N$, then the following properties hold true.
 \begin{enumerate}[label=(\roman*)]
 \item \label{lambdaq}
   There exists at least one $q \in \R^N$ such that $F_0(q)\in \overline G^\alpha(q^\a,p^\a)$ for all $\alpha=1,\dots,N$. The value $F_0(q)$ is denoted by  $\lambda_q$. Moreover, all the $q$ satisfying the previous property satisfy $q\le p$.
 \item \label{lambda}
   The value $\lambda_q$ defined above is independent on $q$. We denote this unique value by 
   $$\lambda=\lambda(p)=:(F_0\overline G)(p).$$
 \end{enumerate}
\end{Proposition}

In order to compose the Godunov semi-fluxes, we have to show that $F_0\underline G $ and $F_0\overline G $ satisfy the same assumptions as $F_0$.
\begin{Proposition}[Properties of $F_0\underline G$ and $F_0\overline G$]
Under the same assumptions,  $F_0\underline G $ and $F_0\overline G $ are continuous, non-increasing and semi-coercive in the sense of  \eqref{eq::p63}.

\end{Proposition}
\begin{proof}
We only do the proof for $F_0\underline G $ since it is similar for $F_0\overline G $.

We begin to show that $F_0 \underline G$ is non-increasing. Let $p_1$ and $p_2$ such that $p_1\ge p_2$ and assume by contradiction that $(F_0\underline G)(p_1)>(F_0\underline G)(p_2).$
Let $q_1,\; q_2$ be such that 
$F_0(q_i)\in\underline G^\a(q_i^\a,p_i^\a)$ for all $\a\in\{1,\dots,N\}.$ In particular, we have $q_1\ge p_1$, $q_2\ge p_2$ and $F_0(q_1)>F_0(q_2)$. We claim that 
$q_1\ge q_2$. Indeed, if for some $\a$, we have $p_2^\a\le p_1^\a\le q_1^\a<q_2^\a$, then 
$$G^\a(q_2^\a,p_2^\a)=F_0(q_2)<F_0(q_1)\le G^\a(q_1^\a,p_1^\a)\le G^\a(q_2^\a,p_2^\a),$$
which is a contradiction. This implies that $q_1\ge q_2$ and so 
$$(F_0\underline G)(p_1)=F_0(q_1)\le F_0(q_2)=(F_0\underline G)(p_2),$$
which is a contradiction. This implies that $F_0\underline G$ is non-increasing.

\bigskip

We now  show that $F_0\underline G $ is continuous. Let $p_n\to p$ and $q_n\ge p_n$ such that 
$$(F_0\underline G )(p_n)=F_0(q_n)\in \underline G^\a(q_n^\a,p_n^\a)\qquad \text{for all }\a\in\{1,\dots,N\}.$$
We claim that $q_n$ is bounded. Indeed, if (up to a subsequence) $q_n^\beta\to +\infty$ for a certain $\beta\in\{1,\dots,N\}$, then $q_n^\beta>p_n^\beta$ and 
$$F_0(p_n)\ge F_0(q_n)= G^\beta(q_n^\beta,p_n^\beta)\ge G^\beta(q_n^\beta,q_n^\beta)=H^\beta(q_n^\beta).$$
Passing to the limit, we then have $F_0(p)=+\infty$, which is absurd. Then $q_n$ is bounded, and up to extract a subsequence, we can assume that $q_n\to q_0\ge p$.

We claim that $F_0(q_0)\in \underline G^\a(q_0^\a,p^\a)$ for all $\a\in\{1,\dots,N\}$. Indeed, fix $\a\in \{1,\dots,N\}$ 
  and assume first that $q_0^\a>p^\a$. We then have $q_n^\a>p_n^\a$ for $n$ large enough and so $F_0(q_n)=G^\a(q_n^\a,p_n^\a)$. Passing to the limit $n\to +\infty$, we get $F_0(q_0)=G^\a(q_0^\a,p^\a)\in\underline G^\a(q_0^\a,p^\a).$ Assume now that $q_0^\a=p^\a$. We distinguish two cases. On the one hand, if there exists a subsequence $n_j$ such that $q_{n_j}^\a=p_{n_j}^\a$, then $F_0(q_{n_j})\le H^\a(q_{n_j}^\a)=H^\a(p_{n_j}^\a)$. Passing to the limit, this implies that $F_0(q_0)\le H^\a(q_0^\a)=H^\a(p^\a)$ and so $F_0(q_0)\in [-\infty,H^\a(p^\a)]=\underline G(q_0,p_0)$. On the other hand, if $q_n^\a>p_n^\a$ for $n$ large enough, then $F_0(q_n)=G^\a(q_n^\a,p_n^\a)$. Again, passing to the limit $n\to +\infty$, we get $F_0(q_0)=G^\a(q_0^\a,p^\a)=H^\a(p^\a)$ and so $F_0(q_0)\in\underline G^\a(q_0^\a,p^\a).$
  
In all the cases, we then have $F_0(q_0)\in \underline G^\a(q_0^\a,p^\a)$. This implies that $(F_0\underline G)(p)=F_0(q_0)$. Since $(F_0\underline G)(p_n)=F_0(q_n)$ and $F_0(q_n)\to F_0(q_0)$, we recover that $(F_0\underline G)(p_n)\to(F_0\underline G)(p)$ and so $F_0 \underline G$ is continuous.

\bigskip

We end the proof showing that $F_0\underline G$ is semi-coercive. We fix $\a\in \{1,\dots,N\}$ and we want to show that if $p^\a\to -\infty$, then $(F_0\underline G)(p)\to+\infty$. 
Let $M>0$. By coercivity of $H^\a$ and semi-coercivity of $F_0$, there exists $p_0^\a$ such that for all $p\in \R^N$ such that $p^\a\le p_0^\a$, we have
$$H^\a(p^\a)\ge M\quad{\rm and}\quad F_0(p)\ge M.$$
Let $p$ be such that $p^\a\le p_0^\a$. Then there exists $q\ge p$ such that $(F_0\underline G)(p)=F_0(q)\in \underline G^\a(q^\a,p^\a)$. We distinguish two cases. If $q^\a>p^\a$, then 
$$(F_0\underline G)(p)=F_0(q)=G^\a(q^\a,p^\a)\ge H^\a(p^\a)\ge M.$$
If $q^\a=p^\a$, then, since $q^\a\le p_0^\a$, we have
$(F_0\underline G)(p)=F_0(q)\ge M$. Then, in all the cases, we have
$$(F_0\underline G)(p)\ge M$$
and so $(F_0\underline G)$ is semi-coercive.
\end{proof}

The result of the previous proposition allows us to compose the action of $\overline G$ with the action of $\underline G$.
We now want to prove that the action of  $G$ on $F_0$ is in fact the action of $\underline G$ on the action of $\overline G$ on $F_0$. More precisely, we have the following result  
\begin{Proposition}[Composition of Godunov semi-fluxes]\label{pro:n3}
Under the same assumptions, we have 
$$(F_0\overline G)\underline G =F_0G=(F_0\underline G)\overline G.$$
\end{Proposition}
In order to prove this proposition,  the following lemma is needed. The proof can be found in \cite[Lemma 5.7]{FIM1}, where the result holds independently for each component $\alpha$.
\begin{Lemma}[Key composition result]\label{lem:n1e}
  \begin{enumerate}[label=(\roman*)]
  \item \label{first}
    For all $q,p \in \R^N$, there exists $\tilde  q \in \R^N$ such that for all $\a\in\{1,\dots,N\}$, $\overline G^\a (q^\a,\tilde q^\a)\cap \underline G^\a(\tilde q^\a,p^\a)\ne \emptyset$. Moreover, for such a vector $\tilde q$, we have 
    $$\overline G^\a (q^\a,\tilde q^\a)\cap \underline G^\a(\tilde q^\a,p^\a)=\{G^\a(q^\a,p^\a) \}\quad  \text{for all }\a\in\{1,\dots,N\}.$$
\item \label{second}
 For all $q,p \in \R^N$, there exists $\tilde  q \in \R^N$ such that for all $\a\in\{1,\dots,N\}$,  $\underline G^\a (q^\a,\tilde q^\a)\cap \overline G^\a(\tilde q^\a,p^\a)\ne \emptyset$. Moreover, for such a vector $\tilde q$, we have 
 $$\underline G^\a (q^\a,\tilde q^\a)\cap \overline G^\a(\tilde q^\a,p^\a)=\{G^\a(q^\a,p^\a) \}\quad  \text{for all }\a\in\{1,\dots,N\}.$$
\end{enumerate}
\end{Lemma}

We are now turn to the proof of Proposition \ref{pro:n3}.
\begin{proof}[Proof of Proposition \ref{pro:n3}]
The proof is similar to the one of \cite[Proposition 5.6]{FIM1}, but for the reader's convenience we give the details. Let $F_1=F_0\overline G$. We use successively the definition of $F_0 G$, \ref{first} from Lemma~\ref{lem:n1e}, the definitions of
$F_0 \overline G$ and of $F_1 \underline G$ to write,
\begin{align*}
  \{ F_0 G (p) \} &= \{ F_0 (q) \text{ for some $q$ s.t. } F_0(q) \in G^\a(q^\a,p^\a)\text{ for all }\a\in\{1,\dots,N\}\} \\
& = \{ F_0 (q) \text{ for some $q$ and $\tilde q$ s.t. } F_0 (q) \in \overline G^\a (q^\a,\tilde q^\a) \cap \underline G^\a(\tilde q^\a,p^\a)\text{ for all }\a\in\{1,\dots,N\}\} \\
  \{ F_1 (\tilde q) \} &= \{F_0 \overline G (\tilde q) \} = \{ F_0 (q) \text{ for some $q$ s.t. } F_0 (q) \in \overline G^\a (q^\a,\tilde q^\a)\text{ for all }\a\in\{1,\dots,N\}\} \\
  \{ F_1 \underline G (p) \} &= \{ F_1 (\tilde q) \text{ for some $\tilde q$ s.t. } F_1 (\tilde q) \in \underline G^\a(\tilde q^\a,p^\a)\text{ for all }\a\in\{1,\dots,N\} \} \\
  & = \{ F_0 (q) \text{ for some $q$ and $\tilde q$ s.t. } F_0 (q) \in \overline G^\a (q^\a,\tilde q^\a) \cap \underline G^\a(\tilde q^\a,p^\a)\text{ for all }\a\in\{1,\dots,N\}\}.
\end{align*}
This implies that $F_0 G (p) = F_1 \underline G (p)= (F_0 \overline G) \underline G(p)$.

Using \ref{second} from Lemma~\ref{lem:n1e}, we can follow the same reasoning and get $F_0 G (p) = (F_0 \underline G) \overline G(p)$.
\end{proof}
\section{Relaxation operator}\label{sec:4}
This section is devoted to the proof of Theorem  \ref{th:relax}. Conversely to the case of a single branch (see \cite{FIM1}), we are not able to show directly that $\underline R \overline R F_0=\overline R \underline R F_0$. To prove this, we will use the link between the relaxation operators and the Godunov's semi-fluxes. This is done in Subsection \ref{subsec:4.1} in the case where $F_0$ is semi-coercive. In Subsection \ref{subsec:4.2} we will give some useful properties of the relaxation operators that will be used to prove Theorem \ref{th:relax} in the general case. Finally, in Subsection \ref{subsec:4.3}, we give further properties of the relaxation operators that will be used in the rest of the paper.

We recall that  the sub-relaxation operator with respect to $H$ is defined by
$$\underline{R} F_0(p) =\sup_{q\ge p} \min\left\{F_0(q),H_{min}(q)\right\}$$
and the super-relaxation operator by
$$\overline{R} F_0(p) =\inf_{q\le p} \max\left\{F_0(q),H_{max}(q)\right\}$$
where $q\le p$ means $q^\alpha\le p^\alpha$ for each $\alpha=1,\dots,N$ and 
$$H_{max}(p)=\max_{\alpha=1,\dots,N}H^\alpha(p^\alpha)\quad\quad  \mbox{and}\quad \quad
H_{min}(p)=\min_{\alpha=1,\dots,N}H^\alpha(p^\alpha).$$

\subsection{Link between Relaxation and Godunov fluxes}\label{subsec:4.1}
The goal of this subsection is to show that the relaxation operator $\frak R$ and the Godunov relaxation coincide when $F_0$ is semi-coercive. We begin by the following proposition.
\begin{Proposition}[Semi-relaxations and Godunov's semi-fluxes]
Let $H$ satisfy \eqref{eq::p00} and  $F_0$  satisfy \eqref{eq::m11} and the semi-coercivity property \eqref{eq::p63}. Then
$$\underline RF_0=F_0\underline G\quad\text{and}\quad \overline RF_0=F_0\overline G.$$
\end{Proposition}
\begin{proof}
We only prove that $\underline RF_0=F_0\underline G$, the proof of $\overline RF_0=F_0\overline G$ being similar. The proof is decomposed into two steps.

\noindent{\bf Step 1: $\underline RF_0\ge F_0\underline G$.}
Let $p\in \R^N$. By Proposition \ref{pro:n2}, there exists $\tilde q$ such that 
$F_0\underline G(p)=F_0(\tilde q)\in\underline G^\a(\tilde q^\a,p^\a)$ for all $\a\in \{1,\dots,N\}$. If $\tilde q^\a>p^\a$, we have
$F_0(\tilde q)=G^\a(\tilde q^\a,p^\a)=\max_{[p^\a,\tilde q^\a]}H^\a$ and we define $\bar q^\a\in [p^\a,\tilde q^\a]$ such that
$$H^\a(\bar q^\a)=\max_{[p^\a,\tilde q^\a]}H^\a=F_0(\tilde q).$$
If $\tilde q^\a=p^\a$, we set $\bar q^\a=p^\a$, so that, using  $F_0(\tilde q)\in \underline G^\a(p^\a,p^\a)$, $F_0(\tilde q)\le H^\a(p^\a)=H^\a(\bar q^\a).$

Using that $\bar q\ge p$, we then have
\begin{equation}\label{eq:002}
\underline R F_0(p)=\sup_{q\ge p}\min (F_0(q),H_{\min}(q))\ge \min (F_0(\bar q), H_{\min}(\bar q)).
\end{equation}
Since $\bar q\le \tilde q$, we have $F_0(\bar q)\ge F_0(\tilde q)$. Moreover, by definition of $\bar q$, we have $H^\a(\bar q^\a)\ge F_0(\tilde q)$. Injecting these two estimates in \eqref{eq:002}, we get
$$\underline R F_0(p)\ge F_0(\tilde q)=F_0 \underline G(p).$$

\noindent{\bf Step 2: $\underline RF_0\le F_0\underline G$.}
With the same notation as in Step 1, we have, for all $q\ge \tilde q$
\begin{equation}\label{eq:003}
\min (F_0(q),H_{\min}(q))\le F_0(q)\le F_0(\tilde q).
\end{equation}
If $\tilde q=p$, we directly get the result, just taking the supremum on $q\ge \tilde q=p$ on the left hand side. We then assume that $\tilde q\ne p$. Hence there exists $\a$ such that $\tilde q^\a>p^\a$. 
Let $q\ge p$ be such that there exists $\a\in\{1,\dots,N\}$ such that $q^\a<\tilde q^\a$. Then
$$\min\left(F_0(q),H_{\min}(q)\right)\le H_{\min}(q)\le H^\a(q^\a)\le \max_{[p^\a,\tilde q^\a]}H^\a=G^\a(\tilde q^\a,p^\a)=F_0(\tilde q).$$
Combining this with \eqref{eq:003}, we finally get
$$\underline R F_0(p)=\sup_{q\ge p}\min (F_0(q),H_{\min}(q))\le F_0(\tilde q)=F_0 \underline G(p).$$
This ends the proof of the proposition.

\end{proof}
A direct consequence of this result, combining with Proposition \ref{pro:n3}, is the following result.
\begin{Proposition}[Definition of $\frak R F_0$ with $F_0$ semi-coercive]\label{pro:RGsemicoercif}
Let $H$ satisfy \eqref{eq::p00} and  $F_0$  satisfy \eqref{eq::m11} and the semi-coercivity property \eqref{eq::p63}. Then
$$\overline R\underline RF_0=\underline R\overline RF_0.$$

\end{Proposition}

The goal of the next section is to generalize this result to the case where $F_0$ doesn't satisfy the semi-coercivity assumption. To do that, we will need some additional properties on the semi-relaxations.
\subsection{First properties of the semi-relaxations and proof of Theorem \ref{th:relax}}\label{subsec:4.2}
We begin by the following lemma whose proof is an immediate consequence of the definition.
\begin{Lemma}[Monotonicity of the semi-relaxations]\label{lem:compR}
Let $H, H'$ satisfy \eqref{eq::p00} and  $F_0, F_0'$  satisfy \eqref{eq::m11}. If $F_0'\le F_0$, then $\underline R F_0'\le \underline R F_0$ and $\overline R F_0'\le \overline R F_0$. Moreover, if we denote by $\underline R_H F_0$ (resp. $\overline R_H F_0$) the sub- (resp. super-) relaxation of $F_0$ with respect to $H$, then if $H\le H'$, then
$$\underline R_H F_0\le \underline R_{H'} F_0\quad{\rm and}\quad  \overline R_{H} F_0\le  \overline R_{H'} F_0.$$
\end{Lemma}

\begin{Lemma}[$\underline R$ and $\overline R$ as projectors]\label{lem:proj}
Let $H$ satisfy \eqref{eq::p00} and  $F_0$  satisfy \eqref{eq::m11}. Then
$$\underline R (\underline R F_0)=\underline R F_0\quad{\rm and}\quad \overline R (\overline R F_0)=\overline R F_0.$$
\end{Lemma}

\begin{proof}
We do only the proof for $\underline R$. We set $F_1=\underline R F_0$. On the one hand, by definition of $\underline R$ and by monotonicity of $F_0$, we have
\begin{equation}\label{eq::N1}
\underline R F_0(p)=\sup_{q\ge p}\min (F_0(q),H_{\min}(q))\le \sup_{q\ge p}F_0(q)=F_0(p).
\end{equation}
Similarly, we get $\underline R F_1(p)\le F_1(p)$.
On the other hand, we have $F_1(p)=\sup_{q\ge p}\min (F_0(q),H_{\min}(q))$ and we assume for simplicity that there exists $q^*\ge p$ such that $F_1(p)=\min (F_0(q^*),H_{\min}(q^*))$. If such point does not exist, we simply consider $\e$-maximizer, which leads classically to the same conclusion in the limit $\e\to 0$.

By definition and monotonicity of $F_1$, we have 
$$F_1(p)\ge F_1(q^*)\ge \min (F_0(q^*),H_{\min}(q^*))=F_1(p).$$
Using \eqref{eq::N1}, we get $F_1(q^*)=\underline R F_0(q^*)\le F_0(q^*)$.  Since $\underline R F_0(q^*)\ge H_{\min}(q^*)$, we get that $$ \min (F_0(q^*),H_{\min}(q^*))=H_{\min}(q^*).$$
Therefore
$H_{\min}(q^*)= F_1(p)=F_1(q^*)$. We then deduce that
$$\underline R F_1(p)\ge \min (F_1(q^*),H_{\min}(q^*))=F_1(p).$$
which implies that $\underline R F_1=F_1$ and ends the proof.
\end{proof}

The following lemma will enable us to replace $F_0$ by a semi-coercive  function.
\begin{Lemma}[Relaxation of $\max(F_0,H_-)$]\label{lem:n1}
Let $H$ satisfy \eqref{eq::p00} and  $F_0$  satisfy \eqref{eq::m11}. Then
$$\overline R (\max(F_0,H_-))=\overline R F_0$$
with $H_-$ defined in \eqref{eq::rep1}.

\end{Lemma}
\begin{proof}
By definition, we have
$$\overline R (\max(F_0,H_-))=\inf_{q\le p}\max(\max(F_0(q),H_-(q)), H_{\max}(q))=\inf_{q\le p}\max(F_0(q),H_-(q), H_{\max}(q))=
\overline R F_0,$$
where we use that $H_-\le H_{\max}$ for the last equality. This ends the proof.
\end{proof}

We are now ready to give the proof of Theorem \ref{th:relax}

\begin{proof}[Proof of Theorem \ref{th:relax}]
We first claim that 
\begin{equation}\label{eq:005}
\overline R \underline RF_0=\overline  R\underline R \max(F_0,H_-).
\end{equation}
Indeed, by definition of $\underline R F_0$, we have
\begin{align*}
\underline R \max(F_0,H_-) (p)
&=\sup_{q\ge p}\min \left\{\max(F_0(q),H_-(q)),H_{\min}(q)\right\}\\
&\le \sup_{q\ge p}\max \left\{\min(F_0(q),H_{\min}(q)),\min (H_-(q),H_{\min}(q))\right\}\\
&\le \max \left\{\sup_{q\ge p}\min(F_0(q),H_{\min}(q)),\sup_{q\ge p}\min (H_-(q),H_{\min}(q))\right\}\\
&\le \max(\underline R F_0(p),\underline R H_-(p)).
\end{align*}
Since 
$$\underline R H_-(p)=\sup_{q\ge p}\min (H_-(q),H_{\min}(q))\le \sup_{q\ge p} H_-(q)=H_-(p),$$
we finally get
$$\underline R \max(F_0,H_-) (p)\le \max(\underline R F_0(p),H_-(p)).$$
Taking the super-relaxation operation, we get (using Lemma \ref{lem:compR})
\begin{equation}\label{eq:007}
\overline R\underline R F_0\le \overline R \underline R\max(F_0,H_-)\le \overline R \max(\underline R F_0(p),H_-(p))=\overline R  \underline R F_0,
\end{equation}
where we use Lemma \ref{lem:n1} for the last equality. This implies \eqref{eq:005}.
Hence, we have
$$\overline R \underline RF_0=\overline  R\underline R \max(F_0,H_-)=\underline  R\overline R \max(F_0,H_-)=\underline R \overline RF_0,$$
where for the second equality we used Proposition \ref{pro:RGsemicoercif} and for the last one, we used Lemma \ref{lem:n1}. 

It just remains to show that $\frak R F_0\ge H_-$. Indeed, by \eqref{eq:007}, we have
$$\frak R F_0=\overline R \underline R F_0=\overline R (\max(\underline R F_0 ,H_-))\ge \overline R H_-\ge H_-,
$$
since
$$\overline R H_-(p)=\inf_{q\le p}\max(H_-(q) ,H_{\max}(q))\ge \inf_{q\le p}H_-(q)=H_-(p).$$
This ends the proof of the theorem.

\end{proof}

\begin{Remark}\label{rem:F-semicoercive}
Note that the fact that $\frak R F_0 \ge H_-$ implies in particular that $\frak R F_0$ is semi-coercive.
\end{Remark}
\subsection{Further properties of the relaxation operator and characteristic points}\label{subsec:4.3}

In this subsection, we give some properties of the relaxation that will be useful in the rest of the paper. We begin by a characterization of the sub-relaxation.

\begin{Proposition}[Characterization of sub-relaxation]\label{pro:6.6}
Let $H$ satisfy \eqref{eq::p00} and  $F_0$  satisfy \eqref{eq::m11}. For $p\in\R^N$, we define
$\bar p\in \R^N$ by
$$\bar p^\a=\left\{\begin{array}{lll}
p^\a&{\rm if}&H^\a(p^\a)\ge F_0(p)\\
\sup\{q^\a\ge p^\a, H^\a(\tilde q^\a)<F_0(p)\quad {\textrm{for all}}\quad \tilde q^\a\in[p^\a,q^\a)\}&{\rm if}&H^\a(p^\a)<F_0(p).
\end{array}
\right.
$$
Then $F_0$ is sub-relaxed, i.e. $F_0=\underline R F_0$, if, and only if, for all $p\in \R^N$, we have
\begin{equation}\label{eq:const}
F_0=const=F_0(p) \quad {\rm in}\quad  [p,\bar p],
\end{equation}
where $[p,\bar p]=\prod_{\a=1}^N [p^\a,\bar p^\a]$.

\end{Proposition}
\begin{proof}
We begin to show that $\underline R F_0=F_0$ implies \eqref{eq:const}. Let $p\in\R^N$ be such that
$$\lambda=F_0(p)>H_{\min}(p)$$
(otherwise $\bar p=p$ and the result is trivial). We claim that
\be\label{eq:201}
\underline RF_0(p)=\sup_{q\ge \bar p}\{\min(F_0(q),H_{\min}(q))\}.
\ee
Indeed, let $q\ge p$ such that there exists $\a\in\{1,\dots,N\}$ such that $q^\a<\bar p^\a$. Then
$$\min(F_0(q),H_{\min}(q)) \le \min(F_0(q),H^\alpha(q^\alpha))=  H^\a(q^\a)<\lambda=F_0(p)=\underline R F_0(p)=\sup_{\tilde q\ge  p}\{\min(F_0(\tilde q),H_{\min}(\tilde q))\}$$
which proves \eqref{eq:201}. We then have
$$\lambda=F_0(p)=\underline R F_0(p)=\sup_{q\ge \bar p}\{\min(F_0(q),H_{\min}(q))\}\le \sup_{q\ge \bar p} F_0(q)=F_0(\bar p)\le F_0(p).$$
Therefore $F_0(\bar p)=F_0(p)$ and since $F_0$ is non-increasing, this implies \eqref{eq:const}.

\bigskip

Conversely, assume that \eqref{eq:const} is satisfied for every $p\in\R^N$.  We then have
$$F_0(p)\ge \underline R F_0(p)=\sup_{q\ge  p}\{\min(F_0(q),H_{\min}(q))\}\ge \min(F_0(\bar p),H_{\min}(\bar p))=F_0(\bar p)=F_0(p)$$
which implies that $F_0=\underline R F_0$ and ends the proof of the proposition.
\end{proof}
A direct consequence of this proposition is the following corollary:
\begin{Corollary}[Lower dimension restriction of sub-relaxed $F_0$ are sub-relaxed]\label{cor:6.7}
Assume that $H$ satisfy \eqref{eq::p00} and  that $F_0$  satisfy \eqref{eq::m11}. For $p\in\R^N$ and $\a \in \{1,\dots,N\}$, we define $F^{\a,p}_0:\R\mapsto \R$ by
$$F^{\a,p}_0(q^\a)=F_0(p^1,\dots,p^{\a-1},q^\a,p^{\a+1},\dots,p^N).$$
If $F_0$ is sub-relaxed with respect to $H$, then $F^{\a,p}_0$ is sub-relaxed with respect to $H^\a$.
\end{Corollary}

We now give the characterization for super-relaxation.
\begin{Proposition}[Characterization of super-relaxation]\label{pro:6.8}
Let $H$ satisfy \eqref{eq::p00} and  $F_0$  satisfy \eqref{eq::m11}. For $p\in\R^N$, we define
$\underline p\in \R^N$ by
$$\underline p^\a=\left\{\begin{array}{lll}
p^\a&{\rm if}&H^\a(p^\a)\le F(p)\\
\inf\{q^\a\le p^\a, H^\a(\tilde q^\a)>F_0(p)\quad {\textrm{for all}}\quad \tilde q^\a\in (q^\a,p^\a]\}&{\rm if}&H^\a(p^\a)>F_0(p).
\end{array}
\right.
$$
Then $F_0$ is super-relaxed, i.e. $F_0=\overline R F_0$, if, and only if, for all $p\in \R^N$, we have
\begin{equation}\label{eq:const1}
F=const=F(p) \quad {\rm in}\quad  [ \underline{p},p],
\end{equation}
and
\begin{equation}\label{eq:underlinePfini}
\underline p^\a>-\infty\quad {\rm if}\quad H^\a(p^\a)>F_0(p).
\end{equation}
\end{Proposition}
\begin{proof}
We begin to show that $\overline R F_0=F_0$ implies \eqref{eq:const1} and \eqref{eq:underlinePfini}. Let $p\in\R^N$ be such that
$$\lambda=F_0(p)<H_{\max}(p)$$
(otherwise $ \underline{p}=p$ and the result is trivial). 
We first show \eqref{eq:underlinePfini}. By contradiction, assume that there exists $\a$ such that $\underline p^\a=-\infty$ and $H^\a(p^\a)>F_0(p)$. By definition of $\underline p^\alpha$, we have $H^\alpha(q^\alpha)>F_0(p)$ for all $q^\alpha\le p^\alpha$.

By coercivity of $H^\a$, there exists $\delta>0$ such that
$$H^\a(q^\a)\ge \delta +\lambda\qquad \textrm{for all } \; q^\a  \le p^\a.$$
We then have
\begin{align*}
\lambda=F_0(p)=\overline R F_0(p)=&\inf_{q\le p}\{\max (F_0(q),H_{\max}(q))\}\\
\ge& \inf_{q\le p}H^\a(q^\a)\ge \delta+\lambda,
\end{align*}
which is absurd. Then \eqref{eq:underlinePfini} is satisfied. We now prove \eqref{eq:const1}. For all $\tilde q\le p$ such that there exists $\a\in\{1,\dots,N\}$ such that $\tilde q^\a>\underline p^\a$, we have
$$\max(F_0(\tilde q),H_{\max}(\tilde q))\ge H^\a(\tilde q^\a)>\lambda=F_0(p)=\overline R F_0(p)=\inf_{q\le p}\max(F_0(q),H_{\max}(q)).$$
Hence
$$\lambda=F_0(p)=\overline R F_0(p)=\inf_{q\le \underline p}\{\max(F_0(q),H_{\max}(q))\}\ge \inf_{q\le \underline p} F_0(q)=F_0(\underline p)\ge F_0(p).$$
Therefore $F_0(\underline p)=F_0(p)$ and since $F_0$ is non-increasing, this implies \eqref{eq:const1}.

\bigskip

Conversely, assume that \eqref{eq:const1} and \eqref{eq:underlinePfini} are satisfied for every $p\in\R^N$.  We then have
$$F_0(p)\le \overline R F_0(p)=\inf_{q\le  p}\{\max(F_0(q),H_{\max}(q))\}\le \max(F_0(\underline p),H_{\max}(\underline p))=F_0(\underline p)=F_0(p)$$
which implies that $F_0=\overline R F_0$ and ends the proof of the proposition.
\end{proof}
A direct consequence of this proposition is the following corollary:
\begin{Corollary}[Lower dimension restriction of super-relaxed $F_0$ are super-relaxed]\label{cor:6.10}
Assume that $H$ satisfy \eqref{eq::p00} and  that $F_0$  satisfy \eqref{eq::m11}. For $p\in\R^N$ and $\a \in \{1,\dots,N\}$, we define $F^{\a,p}_0:\R\mapsto \R$ by
$$F^{\a,p}_0(q^\a)=F_0(p^1,\dots,p^{\a-1},q^\a,p^{\a+1},\dots,p^N).$$
If $F_0$ is super-relaxed with respect to $H$, then $F^{\a,p}_0$ is super-relaxed with respect to $H^\a$.
\end{Corollary}

%

We now give an important result concerning the relaxation at crossing points.
\begin{Proposition}[Relaxation at crossing points ]\label{pro:6.20}
Let $H$ satisfy \eqref{eq::p00} and  $F_0$  satisfy \eqref{eq::m11}. For $p\in\R^N$, assume that
$$H^\a(p^\a)=F_0(p)\quad \textrm{for all } \alpha\in \{1,\dots,N\}.$$
Then 
$$\frak R F_0(p)=\underline R F_0(p)=\overline R F_0 (p)= F_0(p).$$

\end{Proposition}
\begin{proof}
We only prove that $\overline R F_0 (p)= F_0(p)$, the other one being similar.
First, by definition of $\overline R$, we have
$$\overline R F_0(p)=\inf_{q\le p} \max\left\{F_0,H_{\max}\right\}(q)\le \max\left\{F_0,H_{\max}\right\}(p)=F_0(p)$$
since $H_{\max}(p)=F_0(p)$.
On the other hand, we have
$$\overline R F_0(p) =\inf_{q\le p} \max(F_0(q),H_{\max}(q))\ge \inf_{q\le p} F_0(q) =F_0(p).$$
This ends the proof of the Proposition.
\end{proof}

\bigskip

We now  introduce the notion of characteristic points which play an important role in the sequel.

\begin{defi}[Characteristic points]\label{defi::50}
\begin{enumerate}[label=(\roman*)]
\item $p\in \R^N$  is a {\emph super-characteristic point} of $F_0$ if there exists $\varepsilon>0$ such that $H^\alpha(p^\alpha)=F_0(p)$ and 
$H^\alpha(q^\alpha)>H^\alpha(p^\alpha)$ for all $q^\alpha\in \left(p^\alpha,p^\alpha +\varepsilon\right)$ and for all $\alpha=1,\dots,N$. We denote by $\overline \chi (F_0)$ the set of super-characteristic points.
\item $p\in \R^N$  is a {\emph sub-characteristic point} of $F_0$ if there exists $\varepsilon>0$ such that $H^\alpha(p^\alpha)=F_0(p)$ and 
$H^\alpha(q^\alpha)<H^\alpha(p^\alpha)$ for all $q^\alpha\in \left(p^\alpha-\e,p^\alpha \right)$ and for all $\alpha=1,\dots,N$. We denote by $\underline{\chi} (F_0)$ the set of sub-characteristic points.
\item The set of all {\emph characteristic points} is denoted by $\chi(F_0)$, i.e. $\chi(F_0):=\overline{\chi}(F_0)\cup \underline{\chi}(F_0)$.
\end{enumerate} 
\end{defi}

The following proposition will be useful for the reduction of test functions.
\begin{Proposition}[Properties of relaxation]\label{pro:pro-relax}
Let $H$ satisfy \eqref{eq::p00} and  $F_0$  satisfy \eqref{eq::m11}. Then
\begin{equation}\label{eq:pro-relax}
\frak R F_0\le \overline R F_0\le F_0 \quad {\rm on}\; \underline{\chi}(\frak R F_0)
\end{equation}
and 
\begin{equation}\label{eq:pro-relax2}
\frak R F_0\ge \underline R F_0\ge F_0 \quad {\rm on}\; \overline{\chi}(\frak R F_0).
\end{equation}
\end{Proposition}
\begin{proof}
We only prove \eqref{eq:pro-relax}, the proof of \eqref{eq:pro-relax2} being similar. Let $p\in \underline{\chi}(\frak R F_0)$. By Theorem \ref{th:relax}, we have
$$\frak RF_0(p)=\underline R \overline RF_0(p)=\sup_{q\ge p}\min(\overline R F_0(q),H_{\min}(q))\le \sup_{q\ge p} \overline R F_0(q)=\overline RF_0(p)$$
where we used the monotonicity of $\overline R F_0$ for the last equality. Note also that this inequality is true for all $p\in \R^N$, and not only for $p\in  \underline{\chi}(\frak R F_0)$.

 It just remains to show that 
 \begin{equation}\label{eq:0010}
 \overline R F_0(p)\le F_0(p).
 \end{equation}
 Assume by contradiction that $F_0(p)<\overline R F_0(p)$. Since $p\in \underline{\chi} (\frak R F_0)$, we have
 $H_{\max}(p)=H_{\min}(p)=\frak R F_0(p)$ and
 $$H_{\max}(p)=\frak R F_0(p)\le \overline R F_0(p)\le \max(F_0(p),H_{\max}(p))= H_{\max}(p).$$
 Hence $\frak R F_0(p)=\overline RF_0(p)$ and $p\in \underline{\chi}(\overline R F_0)$. This implies that there exists $\e>0$ such that
 $$H_{\max}(q)<\overline R F_0(p)\quad \text{for all }q\in\prod_{\alpha=1,\dots,N} (p^\alpha-\e,p^\alpha). 
 $$By continuity of $F_0$, we can find $\bar q\in  \prod_{\alpha=1,\dots,N} (p^\alpha-\e,p^\alpha)$ such that $F_0(\bar q)<\overline R F_0(p)$. We then have
 $$F_0(\bar q)<\overline R F_0(p)\quad {\rm and}\quad H_{\max}(\bar q)<\overline R F_0(p).$$
Using the definition of $\overline R F_0$, we  get
$$\overline R F_0(\bar q)=\inf_{q\le \bar q}\max(F_0(q),H_{\max}(q))\le \max(F_0(\bar q), H_{\max}(\bar q))<\overline R F_0(p),$$
which contradicts the fact that $\overline R F_0$ is non-increasing. This proves \eqref{eq:0010} and ends the proof of the proposition.
\end{proof}

\begin{Lemma}[Constant sub-relaxed and super-relaxed $F_0$ in a box]\label{lem:const}
Let $H$ satisfy \eqref{eq::p00} and  $F_0$  satisfy \eqref{eq::m11}. If $\underline R F_0(\underline p_0)>H_{\max}(\underline p_0)$ for some $\underline p_0\in \R^N$, then there exists $\underline p_1> \underline p_0$ (i.e. such that  $\underline p_1^\a>\underline p_0^\a$ for all $\a\in\{1,\dots,N\}$), such that
$$\underline p_1\in \underline{\chi}(\underline RF_0)\quad{\rm and}\quad \underline R F_0(\underline p_0)=\underline RF_0(\underline p_1).$$
 In the same way, if $\overline R F_0(\overline p_0)<H_{\min}(\overline p_0)$ for some $\overline p_0\in \R^N$, then there exists $\overline p_1< \overline p_0$ (i.e. such that $\overline p_1^\a<\overline p_0^\a$ for all $\a\in\{1,\dots,N\}$),  such that
$$\overline p_1\in \overline{\chi}(\overline RF_0)\quad{\rm and}\quad \overline R F_0(\overline p_0)=\overline RF_0(\overline p_1).$$

\end{Lemma}

\begin{proof}
We just show the first part of the lemma, since the second one can be prove in a similar way (using moreover the general fact that $\overline R F_0\ge H_-$). For $\a\in\{1,\dots,N\}$, we set
$$\underline p_1^\a=\sup\{p^\a\ge \underline p_0^\a,\; H^\a(q^\a)<\underline RF_0(\underline p_0) \; \forall q^\a\in [\underline p_0^\a, p^\a]\}.$$
Since $H^\a$ is coercive, $\underline p_1^\a$ is well defined. Moreover, by continuity, we have 
$$H_{\min}(\underline p_1)=H^\a(\underline p_1^\a)=\underline R F_0(\underline p_0)\quad{\rm and} \quad H^\a<\underline R F_0(\underline p_0)\; {\rm on}\; [\underline p_0^\a,\underline p_1^\a).$$
Since $\underline R F_0(\underline p_0)>H_{\max}(\underline p_0)\ge H^\a(\underline p_0^\a)$ for all $\a$, we also have $\underline p_1^\a>\underline p_0^\a$. 

Using that $\underline R (\underline R F_0)=\underline R F_0$, we get
\begin{equation}\label{eq:0011}
\underline R F_0(\underline p_0)=\sup_{q\ge \underline p_0}\min (\underline R F_0(q),H_{\min}(q)).
\end{equation}
Now let $q\ge \underline p_0$ such that $q^\a<\underline p_1^\a$ for some $\a$. We then have 
$$H_{\min}(q)\le H^\a(q^\a)<\underline R F_0(\underline p_0)=H_{\min}(\underline p_1).$$
This implies, for all such $q$, that $\min(\underline R F_0(q),H_{\min}(q))<H_{\min}(\underline p_1)=\underline R F_0(\underline p_0)$. Hence, by \eqref{eq:0011}, we get
$$\underline R F_0(\underline p_0)=\sup_{q\ge \underline p_1}\min (\underline R F_0(q),H_{\min}(q))=\underline R F_0(\underline p_1).$$
In particular, $\underline p_1\in \underline{\chi}(\underline RF_0)$.
This ends the proof of the lemma.
\end{proof}

\section{Viscosity solutions and relaxation}\label{sec:5}
\subsection{Weak and strong solutions}

We now introduce our two notions of viscosity solutions: weak and strong viscosity solutions.
Every strong solution will be a weak solution. So, if not specified, any solution can be understood as a weak solution.

For $T>0$, we set 
$J_T=(0,T)\times J$  and we consider the class of test functions on $J_T$
$$C^1(J_T)=\left\{\varphi\in C(J_T), \quad \mbox{the restriction of $\varphi$ to $(0,T)\times J_\alpha$ is $C^1$ for $\alpha=1,\dots,N$}\right\}.$$

We recall the definition of upper and lower semi-continuous envelopes $u^*$ and $u_*$ of a function $u$ defined on $\left[0,T\right)\times J$,
$$u^*(t,x)=\limsup_{(s,y)\to (t,x)} u(s,y) \quad \mbox{and}\quad u_*(t,x)= \liminf_{(s,y)\to (t,x)} u(s,y).$$


\begin{defi}[Weak viscosity solution]\label{defi::1}
Assume that $H$, $F_0$ and $u_0$ satisfy respectively \eqref{eq::p00}, \eqref{eq::m11} and \eqref{eq::p22}, and let $u: [0,+\infty)\times J\to \R$. We say that $u$ is a  {\emph weak $F_0$-subsolution} (resp. weak $F_0$-supersolution) of \eqref{eq::m10}
in $(0,T)\times J$ if $u^*$ is locally bounded from above (resp. $u_*$ is locally bounded from below) and for all test function $\varphi\in C^1(J_T)$ touching $u^*$ from above (resp. $u_*$ from below) at $(t_0,x_0)\in J_T$, we have
$$\varphi_t + H_\alpha(\varphi_x)\le 0 \quad (\mbox{resp.}\quad \ge 0)\quad \mbox{at}\quad (t_0,x_0)\quad {\rm if}\; x_0\in J_\alpha^*$$
and
\begin{equation}\label{eq::p30}
\begin{array}{llll}
&\varphi_t + \min \left\{F_0(\varphi_x), H_{min}(\varphi_x)\right\}&\le 0 &\quad \mbox{at}\quad (t_0,x_0),\\
(\mbox{resp.}\quad &\varphi_t + \max \left\{F_0(\varphi_x), H_{max}(\varphi_x)\right\}&\ge 0 &\quad \mbox{at}\quad (t_0,x_0)),
\end{array}
\end{equation}
if $x_0=0$.\medskip

We say that $u$ is a weak $F_0$-subsolution (resp. weak $F_0$-supersolution) of \eqref{eq::m10}, \eqref{eq::p21} on $[0,T)\times J$ if additionally
$$u^*(0,x)\le u_0(x)\quad (\mbox{resp.} \quad u_*(0,x)\ge u_0(x))\quad \mbox{for all}\quad x\in J.$$
\medskip

We say that $u$ is a {\emph weak $F_0$-solution} if $u$ is both a weak $F_0$-subsolution and a weak $F_0$-supersolution.
\end{defi}

Notice that the first line of \eqref{eq::p30} means for subsolutions that either the desired junction subsolution inequality is satisfied,
or the PDE inequality is satisfied for at least one branch $J_\alpha$. The remark is similar for supersolutions.
A good property of this notion of relaxed solutions is its stability.
For instance the limit of a sequence of weak $F_0$-subsolutions (resp. weak $F_0$-supersolutions) $u_\varepsilon$ is still a weak $F_0$-subsolution
(resp. weak $F_0$-supersolution).\medskip

Another application of the notion of relaxed solution is an existence result, 
easily obtained adapting Perron's method, so we skip the proof (indeed the interested reader may look at \cite[Theorem 2.14]{IM1} where the proof only uses the monotonicity of $F_0$ and the continuity of the Hamiltonians):
\begin{theo}[Existence of relaxed solutions]\label{th::p40}
Assume that $H$, $F_0$, $u_0$ satisfy respectively \eqref{eq::p00}, \eqref{eq::m11} and \eqref{eq::p22}.
Then there exists a relaxed solution $u$ to \eqref{eq::m10} with initial data \eqref{eq::p21}.
\end{theo}

We now give a second definition which requires more on the solution.
\begin{defi}[Strong viscosity solution]\label{defi::2}
The definition of strong $F_0$-subsolutions, strong $F_0$-supersolutions and strong $F_0$-solutions
is exactly the same word by word as in Definition \ref{defi::1}, except that we replace condition \eqref{eq::p30}
by the following one
\begin{equation}\label{eq::p30bis}
\begin{array}{llll}
&\varphi_t +F_0(\varphi_x) &\le 0 &\quad \mbox{at}\quad (t_0,x_0),\\
(\mbox{resp.}\quad &\varphi_t +F_0(\varphi_x)&\le 0 &\quad \mbox{at}\quad (t_0,x_0)).
\end{array}
\end{equation}
\end{defi}

From the definition itself, we see that any strong solution is a weak solution, but the converse is false in general.
We will show in the next subsections that the reverse is true if the junction condition $F_0$ is self-relaxed, i.e. satisfies $F_0=\frak R F_0$.

\subsection{Weak continuity condition at the junction}
We now introduce  the weak continuity condition that will play an important role for reducing the set of test function.
We say that $u$ satisfies  the weak continuity condition if 
\begin{equation}\label{eq:weakcontinuity}
u(t,0)=\limsup_{(s,y)\to (t,0),\ y\in J_\alpha^*}u(s,y)\quad \mbox{for all}\; t>0 \; \mbox { and for each}\; \alpha=1,\dots,N.
\end{equation}

Here the choice of the $\limsup$ (instead of the $\liminf$) is due to the fact that the Hamiltonians $H^\alpha$ are coercive.

We now state the following result whose proof is done in  \cite{IM1} (see there Lemma 2.3, where the proof does not use other properties than  the coercivity of the Hamiltonians to bound the gradient term, and the semi-coercivity of the junction condition to get a contradiction
with a possible discontinuity of the subsolution at the junction):

\begin{Lemma}["Weak continuity" condition at the junction point; \cite{IM1})]\label{lem:semicoercivetoweakcontinuity}
Let $H$ satisfy condition \eqref{eq::p00},
and let $F_0$ satisfy \eqref{eq::m11} and the semi-coercivity condition \eqref{eq::p63}.
Let $u$ be a weak subsolution to \eqref{eq::m10}. Then $u$ satisfies for all $t>0$
$$u^*(t,0)=\limsup_{(s,y)\to (t,0),\ y\in J_\alpha^*} u(s,y)\quad \mbox{for each}\quad \alpha\in \left\{1,\dots,N\right\}.$$
\end{Lemma}

\subsection{Reducing the set of test functions}

We consider functions satisfying a Hamilton-Jacobi equation on $J\backslash \left\{0\right\}$:
\begin{equation}\label{eq::e27}
u_t + H^\alpha(u_x)=0 \quad \mbox{for}\quad (t,x)\in (0,T)\times J_\alpha^*.
\end{equation}

\begin{Proposition}[Reducing the set of test functions for subsolutions]\label{pro::e36}
Assume that $H$ satisfies \eqref{eq::p00} and that $F_0$  satisfies \eqref{eq::m11}.
For any 
$p\in \underline{\chi} (\underline R F_0)$,
let us fix a time independent test function $\phi_p(x)$ satisfying
$\partial_\alpha \phi_p(0)=p^\alpha$.
We then consider the class of test functions of the form
\begin{equation}\label{eq::e38}
\varphi(t,x)=\psi(t) + \phi_p(x)
\end{equation}
with $\phi_p$ fixed for each $p$ as above and $\psi$ a $C^1$ function of time.
Let $u$ be  a function $u: (0,T)\times J \to \R$, upper semi-continuous which is a subsolution of \eqref{eq::e27}. 
Given $t_0\in (0,T)$, we assume that
$$u(t_0,0)=\limsup_{(s,y)\to (t_0,0),\ y\in J_\alpha^*}u(s,y)\quad \mbox{for each}\quad \alpha=1,\dots,N.$$
If for any test function of the class \eqref{eq::e38}, touching $u$ from above at $(t_0,0)$
we have
\begin{equation}\label{eq::e43}
\varphi_t + \underline R F_0(\varphi_x)\le 0
\end{equation}
then $u$ is a strong $\underline R F_0$-subsolution at $(t_0,0)$.
\end{Proposition}

We will use the following result whose proof is done in  \cite[Lemma 2.10]{IM1} (where the proof does not use other properties than the continuity of the Hamiltonians and the coercivity of the Hamiltonians to bound the gradient term):

\begin{Lemma}[Subsolution property for the critical slopes on each branch; \cite{IM1}]\label{lem::e39}
Let $u : (0,T)\times J_\alpha \to \R$ be an upper semi-continuous subsolution of \eqref{eq::e27} for some $\alpha\in \left\{1,\dots,N\right\}$.
Let $\phi$ be a test function touching $u$ from above at some point $(t_0,0)$ with $t_0\in (0,T)$. Consider the following critical slope
$$\underline{p}^\alpha=\inf\left\{ p^\a\in \R, \ \exists r>0,\ \phi(t,x)+ p^\a x \ge u(t,x) \mbox{ for }(t,x)\in (t_0-r,t_0+r)\times [0,r)\mbox{ with } x\in J_\alpha\right\}$$
If 
$$u(t_0,0)=\limsup_{(s,y)\to (t_0,0),\ y\in J_\alpha^*}u(s,y)$$
then $\underline p^\alpha> -\infty$, and we have
$$\phi_t + H^\alpha(\partial_\alpha \phi + \underline p^\alpha)\le 0 \quad \mbox{at}\quad (t_0,0) \quad \mbox{with}\quad \underline p^\alpha\le 0.$$
\end{Lemma}

\begin{proof}[Proof of Proposition \ref{pro::e36}]
Let $u$ be a subsolution of \eqref{eq::e27} such that for any test function of the class \eqref{eq::e38}, touching $u$ from above at $(t_0,0)$, \eqref{eq::e43} holds. Let $\phi$ be a test function touching $u$ from above at $(t_0,0)$. 
We want to show that
\begin{equation}\label{eq::e41}
\phi_t(t_0,0)+ \underline RF_0 (\phi_x(t_0,0))\le 0.
\end{equation}
Notice that, by Lemma \ref{lem::e39}, there exists $-\infty<\underline p^\alpha\le 0$ for each $\alpha=1,\dots,N$ such that
\begin{equation}\label{eq::e40}
\phi_t(t_0,0)+H^\alpha(\partial_\alpha \phi + \underline p^\alpha)\le 0 \quad \mbox{at}\quad (t_0,0).
\end{equation}
We set $p_0=\phi_x(t_0,0)$ and $\underline p_0 = p_0 +\underline p\le p_0$.
If $\underline R F(p_0)\le H_{\max}(\underline p_0)$, then \eqref{eq::e40} implies 
$$\phi_t(t_0,0)+\underline R F_0(p_0)\le \phi_t(t_0,0)+H_{\max}(\underline p_0)\le 0,$$
which gives the desired result.

We then assume that $\underline R F_0(p_0)> H_{\max}(\underline p_0)$. We then have
$\underline R F_0(\underline p_0)\ge \underline RF_0(p_0)>H_{\max}(\underline p_0)$, and Lemma \ref{lem:const}
 implies the existence of some $\underline p_1\ge \underline p_0$, with $\underline p_1^\a>\underline p_0^\a$
such that $\underline R F_0(\underline p_0)=\underline R F_0(\underline p_1)$ and $\underline p_1\in \underline{\chi}(\underline R F_0)$.
Since $\underline p_1^\a>\underline p_0^\a=p_0^\a+\underline p^\a$ for all $\a$, we have by definition of the critical slope $\underline p^\a$ that
$$\varphi(t,x):=\phi(t,0)+\phi_{\underline p_1}(x)  \ge u(t,x)$$
in a neighborhood of $(t_0,0)$ with equality at $(t_0,0)$ and with $\phi_{\underline p_1}$ satisfying
$\partial_\alpha \phi_{\underline p_1}(0)=\underline p_1^\alpha$. Then $\varphi$ is a test  function of the class \eqref{eq::e38}, touching $u$ from above at $(t_0,0)$. By assumption, we then have
$$0\ge \varphi_t (t_0,0)+ \underline R F_0(\underline p_1) = \phi_t (t_0,0) + \underline R F_0(\underline p_0)\ge\phi_t (t_0,0) +  \underline R F_0(p_0)$$
which shows \eqref{eq::e41}. This ends the proof of the proposition.
\end{proof}

As far as strong supersolutions are concerned, it is not necessary to impose a weak continuity assumption, and we show similarly the following result.
\begin{Proposition}[Reducing the set of test functions for supersolutions]\label{pro::e26}
Assume that $H$ satisfies \eqref{eq::p00} and that $F_0$  satisfies \eqref{eq::m11}.
For any 
$p\in \overline{\chi}( \overline R F_0)$,
let us fix a time independent test function $\phi_p(x)$ satisfying
$\partial_\alpha \phi_p(0)=p^\alpha$.
We then consider the class of test functions of the form
\begin{equation}\label{eq::e28}
\varphi(t,x)=\psi(t) + \phi_p(x)
\end{equation}
with $\phi_p$ fixed for each $p$ as above and $\psi$ a $C^1$ function of time.
Let $u$ be  a function $u: (0,T)\times J \to \R$, lower semi-continuous which is a supersolution of \eqref{eq::e27}. 
Given $t_0\in (0,T)$, if for any test function of the class \eqref{eq::e28}, touching $u$ from below at $(t_0,0)$
we have
\begin{equation}\label{eq::e33}
\varphi_t + \overline R F_0(\varphi_x)\ge 0
\end{equation}
then $u$ is a strong $\overline R F_0$-supersolution at $(t_0,0)$.
\end{Proposition}

\subsection{Weak $F_0$-solutions are strong $\frak R F_0$-solutions}
In this section, we will show that $u$ is a weak $F_0$-solution if and only if $u$ is a strong $\frak R F_0$-solution. This result justifies the introduction of the relaxation operator. We begin by the following lemma.

\begin{Lemma}[Weak $F_0$ sub/supersolutions are strong $\underline R F_0/\overline RF_0$ sub/supersolutions]
  \label{lem:relax-op7}
Assume that $H$ satisfies \eqref{eq::p00} and that $F_0$  satisfies \eqref{eq::m11}.
\begin{enumerate}[label=(\roman*)]
\item \label{1romain} Let $u$ be upper semi-continuous.
Then  $u$ is a weak $F_0$-subsolution of \eqref{eq::m10} if and only if $u$  is a strong $\underline R F_0$-subsolution  of \eqref{eq::m10}.
\item \label{2romain} Let $u$ be lower semi-continuous.
  Then $u$ is a weak $F_0$-supersolution  of \eqref{eq::m10}  if and only if $u$ is a strong $\overline R F_0$-supersolution  of \eqref{eq::m10}.
\end{enumerate}
\end{Lemma}
\begin{proof}
We only do the proof for subsolutions since the case of supersolutions is treated similarly. \medskip

\noindent{\bf Step 1: weak implies strong}.
Let $u$ be a weak $F_0$-subsolution of \eqref{eq::m10} and $\phi$ be a test function touching $u$ from above at $(t_0,0)$. Let $\overline q\in \R^N$, $\overline q\ge 0$. We set
$$\psi(t,x)=\phi(t,x)+\overline q^\a x\quad {\rm if}\; x\in J_\a.$$
In particular $\psi$ is a test function touching $u$ from above at $(t_0,0)$. Hence
$$\psi_t(t_0,0)+\min\left\{F_0,H_{\min}\right\}(\psi_x(t_0,0))\le 0$$
i.e., for all $\overline q\ge 0$, we get
$$\phi_t(t_0,0) +\min \left\{F_0,H_{\min}\right\}(\phi_x(t_0,0)+\overline q)\le 0.$$
Taking the supremum over $q:=\phi_x(t_0,0)+\overline q$, we finally get
$$\phi_t(t_0,0) +\sup_{q\ge\phi_x(t_0,0)} \min\left\{F_0,H_{\min}\right\} (q)\le 0$$
and so $u$ is a strong $\underline R F_0$-subsolution.

\medskip

\noindent{\bf Step 2: strong implies weak}.
Let $u$ be a strong $\underline R F_0$-subsolution of \eqref{eq::m10} and $\phi$ be a test function touching $u$ from above at $(t_0,0)$. Hence
$$\phi_t(t_0,0)+\min\left\{F_0,H_{\min}\right\}(\phi_x(t_0,0))\le \phi_t(t_0,0)+\sup_{q\ge \phi_x(t_0,0)} \min\left\{F_0,H_{\min}\right\}(q)\le 0,$$
which implies that $u$ is a weak $F_0$-subsolution.
This ends the proof of the lemma.

\end{proof}

Since we may have $\underline R F_0< \overline R F_0$, Lemma \ref{lem:relax-op7} is not completely satisfactory,  because we would like to have the same boundary function for sub- and supersolutions. This is achieved in the following results where the common boundary function is $\mathfrak R F_0$.

\begin{Proposition}[Weak $F_0$ sub/supersolutions are strong $\frak RF_0$ sub/supersolutions]\label{pro:wtostr}
Assume that $H$ satisfies \eqref{eq::p00} and that $F_0$  satisfies \eqref{eq::m11}.
Let $u$ be an upper semi-continuous function. 
\begin{enumerate}[label=(\roman*)]
\item If $u$ is a weak $F_0$ subsolution of \eqref{eq::m10} and if $u$ satisfy the weak continuity condition \eqref{eq:weakcontinuity}, then $u$ is a strong $\frak R F_0$-subsolution of \eqref{eq::m10}.
\item If $u$ is a strong $\frak R F_0$-subsolution of \eqref{eq::m10}, then $u$ is a weak $F_0$-subsolution of \eqref{eq::m10}.
\end{enumerate}
Let $v$ be a lower semi-continuous function. 
\begin{enumerate}[label=(\roman*)]
\item If $v$ is a weak $F_0$ supersolution of \eqref{eq::m10}, 
then $v$ is a strong $\frak R F_0$-supersolution of \eqref{eq::m10}.
\item If $v$ is a strong $\frak R F_0$-supersolution of \eqref{eq::m10}, then $v$ is a weak $F_0$-supersolution of \eqref{eq::m10}.
\end{enumerate}
\end{Proposition}

\begin{proof}
We only do the proof for the subsolution. We set $F:=\frak R F_0$. Let $u$ be a weak $F_0$ subsolution of \eqref{eq::m10}satisfying the weak continuity condition \eqref{eq:weakcontinuity} and let $\phi$ be a test function touching $u$ from above at $(t_0,0)$ with $t_0>0$. We set $\lambda=\phi_t(t_0,0)$ and $p=\phi_x(t_0,0)$. We then have
$$\lambda+\min(F_0(p),H_{\min}(p))\le0.$$
Since, by Lemma \ref{lem:proj}, we have $F=\underline R F$, we know from Proposition \ref{pro::e36} that we can assume that $p\in \underline{\chi} F$. By Proposition \ref{pro:pro-relax}, we then have
$H_{\min}(p)=F(p)=\frak R F_0(p)\le F_0(p)$ and so $\lambda+\frak R F_0(p)\le0$. This implies that $u$ is a strong $\frak R F_0$-subsolution.\medskip

We now assume that $u$ is a strong $F$-subsolution. Since $F=\overline R \underline R F_0\ge \underline R F_0$, we deduce that $u$ is also a strong $\underline R F_0$-subsolution. By Lemma \ref{lem:relax-op7} \ref{1romain}, we then deduce that $u$ is a weak $F_0$-subsolution.
\end{proof}

\subsection{Different junction conditions leading to the same problem}\label{s5.5}

In this short subsection, we give a simple example limited to a single branch ($N=1$), in order to simplify the presentation. We will introduce different boundary conditions (i.e. junction conditions here).
We consider the convex Hamiltonian $H:\R\to \R$ given by $H(p):=|p|$, and define its nonincreasing envelope $H^-(p):=\max(0,-p)$. We consider the following equation
\begin{equation}\label{eq::R1}
\left\{\begin{array}{ll}
u_t+ H(u_x)=0&\quad \mbox{on}\quad (0,+\infty)_t\times (0,+\infty)_x\\
u=u_0& \quad \mbox{on}\quad \left\{0\right\}_t\times [0,+\infty)_x\\
\end{array}\right.
\end{equation}
where the initial data $u_0$ is assumed to be Lipschitz continuous.

Given a flux limiter, i.e. a parameter $A\ge 0=\min H$, we consider the following boundary condition
\begin{equation}\label{eq::R2}
u_t + \max\left\{A,H^-(u_x)\right\}=0 \quad \mbox{on}\quad (0,+\infty)_t\times \left\{0\right\}_x.
\end{equation}
Now we consider the following three other boundary conditions
\begin{equation}\label{eq::Ra}
u_t + A=0\quad \mbox{on}\quad (0,+\infty)_t\times \left\{0\right\}_x
\end{equation}
or
\begin{equation}\label{eq::Rb}
u_t + 2A-u_x=0\quad \mbox{on}\quad (0,+\infty)_t\times \left\{0\right\}_x
\end{equation}
or
\begin{equation}\label{eq::Rc}
-u_x=-A\quad \mbox{on}\quad (0,+\infty)_t\times \left\{0\right\}_x.
\end{equation}

Then we have the following result.
\begin{Lemma}\label{lem::R5}{\bf (Same effective boundary condition)}\\
Under the previous assumptions and say, for continuous weak solutions $u:[0,+\infty)^2\to \R$, the four following problems are equivalent:
problem \eqref{eq::R1},\eqref{eq::R2}, problem \eqref{eq::R1},\eqref{eq::Ra}, problem \eqref{eq::R1},\eqref{eq::Rb} and  problem \eqref{eq::R1},\eqref{eq::Rc}.
\end{Lemma}

\begin{Remark}\label{rem::R8}
At the boundary $x=0$, recall that weak subsolutions satisfy  either the boundary subsolution inequality or the PDE subsolution inequality (and similarly for supersolutions).
\end{Remark}

\begin{proof}
Each desired boundary condition \eqref{eq::R2}, \eqref{eq::Ra}, \eqref{eq::Rb} writes 
\begin{equation}\label{eq::R7}
u_t+F_0(u_x)=0.
\end{equation}
Moreover the point $A$ satisfies $\chi (F_0)=\left\{A\right\}$.
It is then easy to check that $\frak R F_0=\min\left\{A,H^-\right\}$ and to conclude to the equivalence.

Condition \eqref{eq::Rc} is slightly different and requires more attention. Let us define the maximal monotone (nonincreasing) graph
$$F_0(p):=\left\{\begin{array}{ll}
\left\{+\infty\right\}& \quad \mbox{if}\quad p<A\\
\left[-\infty,+\infty\right]& \quad \mbox{if}\quad p=A\\
\left\{-\infty\right\}& \quad \mbox{if}\quad p>A\\
\end{array}\right.$$
Then subsolutions and supersolutions of \eqref{eq::Rc} can be rewritten at the junction point $x=0$ as
$$\left\{\begin{array}{ll}
\mbox{there exists $\overline \mu \in F_0(u_x)$ such that $u_t+\overline \mu \ge 0$}\quad \mbox{\bf (supersolution)}\\
\mbox{there exists $\underline \mu \in F_0(u_x)$ such that $u_t+\underline \mu \le 0$}\quad \mbox{\bf (subsolution).}\\
\end{array}\right.$$
This formulation is now closer to formulation \eqref{eq::R7}.
Notice also that such $F_0$ can be seen as a certain limit as $\varepsilon\to 0$ of $F_\varepsilon(p)=A-\varepsilon^{-1}(p-A)$.
Moreover the point $A$ satisfies 
$\left\{p\in \R,\ H(p)\in F_0(p)\right\}=\left\{A\right\}$ (and $\chi(F_\varepsilon)=\left\{A\right\}$).
It is then easy to check that $\frak R F_\varepsilon=\min\left\{A,H^-\right\}$. With a certain abuse of definition, we can also write $\chi(F_0)=\left\{A\right\}$ and $\frak R F_0=\min\left\{A,H^-\right\}$. Then along the same lines (and skipping the technicalities specific to that case), it is also possible to check that problem \eqref{eq::R1},\eqref{eq::Rc} is equivalent to problem \eqref{eq::R1},\eqref{eq::R2}. This ends the proof of the lemma.
\end{proof}

\section{Comparison principle}\label{sec:6}
\begin{Theorem}[Comparison principle]\label{th:PC}
Assume that $H$ satisfies \eqref{eq::p00} and that $F_0$  satisfies \eqref{eq::m11} and \eqref{eq::p63}. 
Assume also that $u_0$ is bounded and Lipschitz continuous. Let $u : [0,T) \times J\to \R$ (resp. $v$) be a bounded upper semi-continuous weak viscosity $F_0$-subsolution (resp. bounded lower semi-continuous weak viscosity $F_0$-supersolution) of \eqref{eq::m10}-\eqref{eq::p21}. If $u(0,\cdot)\le u_0 \le v(0,·)$ in $J$, then 
$$u\le v \quad {\rm in}\; [0,T)\times J.$$

\end{Theorem}

\begin{Remark}
{Notice that in Theorem \ref{th:PC}, semi-coercivity of $F_0$ in condition \eqref{eq::p63} can be replaced by the weak continuity of the subsolution $u$ at the junction, using Proposition \ref{pro:wtostr} and replacing $F_0$ by $\frak R F_0$.}
\end{Remark}

The proof of the theorem is very closed to the one of \cite[Theorem 1]{FIM2} (see also \cite{FIM3}) and we skip it. The key point is to replace \cite[Corollary 1]{FIM2} by the following proposition which proof is a very simple extension.
\begin{Proposition}[Junction viscosity inequalities]\label{pro::n6jb}
We consider two sets of functions $u,v: J \to \R\cup \left\{-\infty,+\infty\right\}$ with $u$ upper semi-continuous and $v$ lower semicontinous satisfying
\begin{equation}\label{eq::n2b}
u(0)=0=v(0)\quad \mbox{with}\quad u\le v \quad \mbox{on}\quad J.
\end{equation}
For $\a=1,\dots,N$, we define
\begin{equation}\label{eq::n12b}
\overline p^\a:=\limsup_{J_\a^*\ni x \to 0} \frac{u(x)}{x},\quad \underline p^\a:=\liminf_{J_\a^*\ni x \to 0} \frac{v(x)}{x}.
\end{equation}
We also set  $a^\a:=\min \left\{\underline p^\a, \overline p^\a\right\}$, $b^\a:=\max\left\{\underline p^\a, \overline p^\a\right\}$ and
$$[a,b]\cap \R^N:=\prod_{\a=1,\dots,N} \left([a^\a,b^\a]\cap \R\right).$$
For $\gamma=1,2$, consider continuous functions $H_\gamma^\a: \R\to \R$ and $F_\gamma:\R^N \to \R$ with $H_1^\a$ coercive and $F_1$  semi-coercive. For $p=(p^1,\dots,p^N)\in \R^N$, we set
$$H_{\gamma; \min}(p)=\min_{\a=1,\dots,N} H^\a_\gamma(p^\a),\quad H_{\gamma;\max}(p)=\max_{\a=1,\dots,N} H^\a_\gamma(p^\a).$$
We also assume that we have the following viscosity inequalities for some $\eta>0$
\begin{equation}\label{eq::n7b}
\left\{\begin{array}{rlrll}
H^\a_1(u_x) \le 0 &\quad  \mbox{on} & \quad J_\a &\cap \left\{|u|< +\infty\right\}& \quad \mbox{for}\quad \a=1,\dots,N\\
\min\left\{F_1,H_{1;\min}\right\}(u_x) \le 0 &\quad  \mbox{on} & \quad \left\{0\right\} &\cap \left\{|u|< +\infty\right\}&\\
\\
H_2^\a(v_x) \ge  \eta &\quad  \mbox{on} & \quad J_\a&\cap \left\{|v|< +\infty\right\}& \quad \mbox{for}\quad \a=1,\dots,N\\
\max\left\{F_2,H_{2;\max}\right\}(v_x) \ge \eta &\quad  \mbox{on} & \quad \left\{0\right\} &\cap \left\{|v|< +\infty\right\}.&\\
\end{array}\right.
\end{equation}
Then  there exists $p=(p^1,\dots,p^N)\in [a,b]\cap \R^N\not=\emptyset$ such that
\begin{equation}\label{eq::n8jb}
\left\{\begin{array}{lrcll}
\mbox{either}&\quad H^\a_1(p^\a)&\le 0 < \eta \le & (H^\a_2-H^\a_1)(p^\a) &\quad \mbox{for some $\a\in \left\{1,\dots,N\right\}$}\\
\mbox{or}&\quad \max(F_1,H_{1;\max})(p)&\le 0 <  \eta \le &(F_2-F_1)(p).&
\end{array}\right.
\end{equation}
\end{Proposition}
\bigskip
%
%
%

\section{A third relaxation formula using Riemann problem}\label{sec:7}
The goal of this section is to give a third relaxation formula. This formula is based on Riemann problem, defined for $p\in \R^N$ by
\begin{equation}\label{eq:riemann}
\left\{\begin{array}{ll}
{u}_t + {H}^\alpha({u}_x)=0 & \mbox{on} \; (0,+\infty)\times J_\alpha^*,\quad \mbox{for}\quad \alpha=1,\dots,N\\
{u}_t + {F_0}(\partial_1 {u}(t,0^+),\dots,\partial_N {u}(t,0^+))=0 & \mbox{on} \; (0,+\infty) \times \{0\},\\
u(0,x)=u_0(x)=px&\mbox{on}\; \left\{0\right\}\times J,
\end{array}\right.
\end{equation}
where the initial condition $u_0(x)=px$ means that
$$u_0(x)=p^\a x\quad \textrm{for all }x\in J_\alpha\quad \mbox{and}\quad \a=1,\dots ,N.$$
We then have the following result concerning the properties of the solution to the Riemann problem.

\begin{Theorem}[Properties of the Riemann problem]\label{th:existence-riemann}
Assume that $H$ satisfies \eqref{eq::p00} and that $F_0$  satisfies \eqref{eq::m11}. Then for every $p\in \R^N$ there exists a unique weak $F_0$-solution $u$ to problem \eqref{eq:riemann} satisfying the weak continuity condition \eqref{eq:weakcontinuity} and such that, for any $T>0$ there exists $C_T>0$ such that
$$|u(t,x)-u_0(x)|\le C_T\quad \textrm{for all }(t,x)\in\{0,T)\times J.
$$
Moreover, $u$ is self-similar, i.e. it satisfies 
$$u(\mu t,\mu x)=\mu u (t,x)\quad \textrm{for all }\mu>0\textrm{ and } (t,x)\in[0,+\infty)\times J$$
and $u$ is globally Lipschitz continuous and convex or concave in $(t,x)$ on each branch $[0,+\infty)\times J_\a$, the convexity/concavity depending on the branch $\a=1,\dots,N$.
\end{Theorem}

\begin{proof}
We set $F=\frak R F_0$, which is semi-coercive because of Remark \ref{rem:F-semicoercive}.

\noindent{\bf Step 1: Existence of a solution to \eqref{eq:riemann}.} Let  
$$C\ge \max(\max_\a |H_\a(p^\a)|, |F(p)|).$$ Then $(t,x)\mapsto u_0(x)+Ct$ and $(t,x)\mapsto u_0(x)-Ct$ are respectively super and subsolution to \eqref{eq:riemann} with $F_0$ replaced by $F$. Using Perron's method, we deduce that there exists a weak $F$-solution $u$ to \eqref{eq:riemann}, satisfying the following barriers
\begin{equation}\label{eq::N2}
u_0-Ct\le u\le u_0+Ct.
\end{equation}
Then the comparison principle implies that $u^*\le u_*$ which shows that $u$ is continuous.
Then, using $\underline R F=F=\overline R F$, Lemma \ref{lem:relax-op7} implies that $u$ is a strong $F$-solution. Now Proposition \ref{pro:wtostr} implies that $u$ is a weak $F_0$-solution.

By Proposition \ref{pro:wtostr}, we deduce that $u$ is also a $F_0$-strong solution and a $F_0$-weak solution.\bigskip

\noindent{\bf Step 2: $u$ is globally Lipschitz continuous.} 
For $h>0$, we  set $u^h(t,x)=u(t+h,x)-Ch$. From \eqref{eq::N2}, we deduce that
$u^h$ is solution to the same equation and satisfies
$$u^h(0,x)=u(h,x)-Ch\le u_0(x).$$
The comparison principle then implies that
$$u(t+h,x)-Ch=u^h(t,x)\le u(t,x).$$
In the same way, we have that 
$$u(t+h,x)+Ch\ge u(t,x),$$
which implies that $u$ is Lipschitz continuous in time. Since $H^\a$ is coercive, the equation satisfied by $u$ implies that $u$ is also globally Lipschitz continuous on each open branch $(0,+\infty)\times J_\a^*$, and then on the whole $[0,+\infty)\times J$.
\bigskip

\noindent{\bf Step 3: $u$ is self-similar and convexe or concave}. Let $\mu>0$ and set $u^\mu(t,x)=\frac 1\mu u(\mu t,\mu x)$. Then $u^\mu$ is solution to the same equation and satisfies
$$u^\mu(0,x)=\frac 1\mu u(0,\mu x)=px=u(0,x).$$
Then, by uniqueness, we get that 
$$u=u^\mu=\frac 1\mu u(\mu t,\mu x)\quad \textrm{for all }\mu>0,$$
which show that $u$ is self-similar.

We now fix $\a\in\{1,\dots, N\}$. Since $u$ is self-similar, there exists $W:J_\a\to \R$ such that 
$$u(t,x)=tW\left(\frac {x} t\right).$$
We want to prove that $W$ is either convex or concave on $J_\a$. 
We have that $W$ is a viscosity solution of
\begin{equation}\label{eq:W}
W-\xi \partial_\a W(\xi)+H^\a(\partial_\a W(\xi))=0\quad \textrm{if }\xi \in J_\a^*.
\end{equation}
For $0<a<b$ with $a,b\in J_\alpha^*$, we define 
$$S^\a_{a,b}:=\left\{(t,x)\in (0,+\infty)\times J_\a;\; a<\frac {x}t<b\right\}.$$
We then set
$$\bar p:=\frac{W(b)-W(a)}{b-a},\quad  \bar \lambda :=\frac{bW(a)-aW(b)}{b-a}$$
and 
$$v(t,x)=\bar \lambda t+\bar p x$$
so that
$$v(t,x)=u(t,x)\quad \textrm{for any } (t,x)\in \partial S^\a_{a,b}.$$
Since $v$ is self-similar, there exists $W_{a,b}$ such that $\displaystyle{v (t,x)=tW_{a,b}\left(\frac x t\right)}$.
We define $\tilde W:=W-W_{a,b}$ which satisfies $\tilde W=0$ on $\partial [a,b]$.
Using that 
$\bar \lambda=v_t(t,x)=W_{a,b}\left(\xi\right)-\xi\partial_\a W_{a,b}\left(\xi\right)$ for $\xi:=\frac{x}{t}$, we deduce that 
$$\tilde W-\xi \partial_\a\tilde W=W-\xi\partial _\a W-\left\{W_{a,b}-\xi\partial _\a W_{a,b}\right\}=-H^\a(\partial _\a W)-\bar \lambda=-\tilde H^\a(\partial _\a \tilde W)$$  
with $\tilde H^\alpha(p):=H^\alpha(p+\bar p)+\bar \lambda.$

We now distinguish two cases. If $\tilde H(0)\ge 0$, we assume by contradiction that
$$\max_{\xi\in J_\a,\; \xi\in[a,b]}\tilde W=\tilde W(\bar \xi)>0\quad \textrm{with }\bar \xi\in (a,b).$$ 
Then the equation satisfied by $\tilde W$ implies that
$$\tilde W(\bar \xi)+\tilde H^\a(0)\le 0$$
which is absurd. Therefore
$$\tilde W=W-W_{a,b}\le 0\quad \textrm{on }(a,b) \subset J_\alpha^*,$$
i.e.
\begin{align}\label{eq:789}
W(\xi) \le & W_{a,b}(\xi)=\bar \lambda +\bar p\xi\nonumber\\
=&\frac{bW(a)-aW(b)}{b-a}+ |\xi| \frac{W(b)-W(a)}{b-a}\nonumber\\
=&W(a)+(\xi-a)\left(\frac {W(b)-W(a)}{b-a}\right):=c_{a,b}(\xi)
\end{align}

In the same way, if  $\tilde H(0)\le 0$, we can show that
\begin{equation}\label{eq:790}
W\ge c_{a,b}(\xi)\quad \textrm{on } (a,b)\subset J_\alpha^*
\end{equation}
We can easily check that this implies that we have either \eqref{eq:789} for all choices $b>a>0$ or we have \eqref{eq:790} for all choices $b>a>0$. This means that $W$ is either convex or concave on $J_\a$ and so $u$ is convex or concave on each branch.
\end{proof}

We are now ready to give the third relaxation formula:
\begin{Theorem}[Relaxation through Riemann problem]
Assume that $H$ satisfies \eqref{eq::p00} and that $F_0$  satisfies \eqref{eq::m11}. For $p\in \R^N$, let $u$ be the unique $F_0$-relaxed solution to problem \eqref{eq:riemann} given by Theorem \ref{th:existence-riemann}.
Then the relaxation operator is given by 
$$\frak R F_0(p)=-u_t(t,0)\quad \textrm{for all }t>0.$$
Moreover the restriction of $u$ on each branch has a derivative at $x=0$ and satisfies
$$u_x(t,0^+)=\hat p  \quad \mbox{for all}\quad t>0$$
where $\hat p\in \R^N$ and is given by (with $F=\frak RF_0$)
$$\hat p^\alpha = \left\{\begin{array}{ll}
p^\alpha & \quad \mbox{if}\quad H^\alpha(p^\alpha)=F(p),\\
\\
\sup \left\{q^\alpha\ge p^\alpha,\quad H^\alpha(q'^\alpha)< F(p) \quad \mbox{for all}\quad q'^\alpha\in [p^\alpha,q^\alpha)\right\}
& \quad \mbox{if}\quad H^\alpha(p^\alpha)<F(p),\\
\\
\inf \left\{q^\alpha\le p^\alpha,\quad H^\alpha(q'^\alpha)> F(p) \quad \mbox{for all}\quad q'^\alpha\in (q^\alpha,p^\alpha]\right\}
& \quad \mbox{if}\quad H^\alpha(p^\alpha)>F(p),\\
\end{array}\right.$$
Moreover $F(\hat p)=F(p)$ and then 
in particular $(F(\hat p),\hat p)$ belongs to the associated germ
$$\mathcal G = \left\{(\lambda,q)\in \R\times \R^N,\quad \lambda= F(q)=H^\alpha(q^\alpha)\quad \mbox{for all}\quad \alpha=1,\dots,N\right\}$$
Therefore, we have
$$F(p)=F(\hat p)=G^\alpha(\hat p^\alpha,p^\alpha)=H^\alpha(\hat p^\alpha)\quad \mbox{for all}\quad \alpha=1,\dots,N$$
\end{Theorem}

\begin{proof}
The idea of the proof is to give another construction of the solution to the Riemann problem. This solution is constructed on each branch separately (Step 1) and glued at the junction (Step 2). By uniqueness, we will get that this solution is in fact equal to the solution $u$ given by Theorem \ref{th:existence-riemann}.\medskip

We fix $p\in \R^N$ and we set $F:=\frak R F_0$.

\noindent{\bf Step 1: Construction on each branch separately}. Given $p\in \R^N$, we fix $\a\in\{1,\dots,N\}$ and we want to construct a solution $w$ to 
\begin{equation}\label{eq:surJalpha}\left\{\begin{array}{lll}
w_t+H^\a(w_x)=0 &\textrm{on} &(0,+\infty)\times J_\a^*\\
w_t+F^\a_p(w_x)=0&\textrm{on}&(0,+\infty)\times \{0\},\\
w(0,x)=p^\a x&\textrm{for all}& x\in J_\a
\end{array}
\right.
\end{equation}
where $F^\a_p(q^\a)=F(p^1,\dots,,p^{\a-1},q^\a,p^{\a+1},\dots ,p^N)$. 
Using the same arguments as in the proof of Theorem \ref{th:existence-riemann}, we deduce that there exists a unique  solution $w$. For $\beta=1,\dots, N$, we define
$$\overline{p}^\beta=\max(p^\beta,\hat p^\beta) \quad \mbox{and}\quad \underline{p}^\beta=\min(p^\beta,\hat p^\beta).$$
Recall that $F$ is relaxed with respect to $H$.
From the characterization of relaxed junction conditions (Proposition \ref{pro:6.6} and \ref{pro:6.8}), we deduce that 
$F$ is constant on the box $[\underline p,\overline p]$. In particular, we get
$$F^\a_p(\hat p^\a)=F^\a_p(p^\a)=F(p).$$
Notice that it can also be seen from Corollaries \ref{cor:6.7} and \ref{cor:6.10} for $F^\a_p$ which is then also relaxed with respect to $H^\a$.
We now distinguish two cases.

\noindent {\bf Case 1: $\hat p^\alpha\le p^\alpha$.} We then have
\begin{equation}\label{eq:compLip}
H^\alpha(p^\alpha)\ge F(p)= F^\alpha_p(p^\alpha)=F_p^\alpha(\hat p^\alpha)=F(\hat p)=H^\alpha(\hat p^\alpha).
\end{equation}
We claim that 
\begin{equation}\label{eq:grad}
\hat p^\alpha \le  w_x \le p^\alpha.
\end{equation}
Since the arguments are rather classical, we juste give the sketch of the proof. Let us prove the first inequality. By contradiction, we assume that 
$$M:=\sup_{\{t\in[0,T],\; x,y\in J_\a,\;x\ge y\}}\{w(t,y)-w(t,x)-\hat p^\a(y-x)\}>0.$$
Hence, for every $\delta>0$ small enough, we have
$$M_{\delta}:=\sup_{\{t\in[0,T],\; x,y\in J_\a,\;x>y\}}\left\{w(t,y)-w(t,x)-\hat p^\a(y-x)-\frac \delta 2(x^2+y^2)\right\}>0.$$
Note that all maximizers $(\bar t, \bar x,\bar y)$ satisfy $\bar x>\bar y$ (otherwise the maximum is negative) and, by classical arguments, $\delta \bar x, \delta \bar y\to 0$ as $\delta\to 0$. We now duplicate the time variable and consider, for $\eta, \e>0$, 
$$M_{\e,\delta}:=\sup_{\{t,s\in[0,T],\; x,y\in J_\a,\; x>y\}}\left\{w(t,y)-w(s,x)-\hat p^\a(y-x)-\frac \delta 2(x^2+y^2)-\frac{|t-s|^2}{2\e}-\frac \eta{T-t}\right\}>0.$$
This maximum is reached at some point denoted by $(t_\e,s_\e,x_\e,y_\e)$ which satisfies $t_\e,s_\e\to \bar t$, $x_\e\to \bar x$ and $y_\e\to \bar y$, where $(\bar t,\bar x,\bar y)$ is a maximum point of $M_{\delta}$. In particular, $x_\e>y_\e$ for $\e$ small enough and we assume that $x_\e>0=y_\e$, the other cases being treated in a similar (and even simpler) way. Using the viscosity inequality satisfied by $w$ (which is a strong $F$-solution to \eqref{eq:surJalpha}), we then have
$$\frac {t_\e-s_\e}\e+\frac {\eta}{(T-t_\e)^2}+F^\a_p(\hat p^\a+ \delta y_\e)\le0$$
and
$$\frac {t_\e-s_\e}\e+H^\a(\hat p^\a-\delta x_\e)\le0.$$
Subtracting these two inequalities, we then get
$$\frac \eta {T^2}\le H^\a(\hat p^\a-\delta x_\e)-F^\a_p(\hat p^\a+ \delta y_\e).$$
Passing to the limit $\e\to 0$ and then $\delta\to 0$, we then obtain
$$\frac  \eta {T^2}\le H^\a(\hat p^\a)-F^\a_p(\hat p^\a)\le 0,$$
where for the last inequality, we used \eqref{eq:compLip}. Contradiction. Hence $M\le 0$ and then $w_x \ge \hat p$.

In the same way (using that $H^\a(p^\a)\ge F^\a_p(p^\a)$), we can prove the second inequality in \eqref{eq:grad}.

\medskip

We now want to prove that 
\begin{equation}\label{eq::z33}
w_x(t,0)= \hat p^\alpha\quad \mbox{and}\quad -w_t(t,0)=F(p).
\end{equation}
Note that since $w$ is convex or concave, these derivatives exist. 
For all $t_0>0$, we then set
$$w_x(t_0,0)= \check p^\alpha\quad \mbox{with}\quad \hat p^\alpha \le \check p^\alpha\le p^\alpha$$
and
$$\lambda=-w_t(t_0,0)=F_p^\alpha(\check p^\alpha)=F(p)=F(\hat p).$$
Making a blow up at $(t_0,0)$, we see that
$$\varepsilon^{-1}\left\{w(t_0+\varepsilon t,\varepsilon x)-w(t_0,0)\right\} \to W(t,x):=-\lambda t + \check p^\alpha x \quad \mbox{as}\quad \varepsilon\to 0.$$
By stability of viscosity solutions, we get that $W$ is still a solution and then satisfies for $x\in J_\a^*$
$$W_t + H^\alpha(W_x)=0$$
i.e.
$$-\lambda + H^\alpha(\check p^\alpha)=0.$$
We conclude that
$$F_p^\alpha(\check p^\alpha)=H^\alpha(\check p^\alpha)=F(p) \quad \mbox{with}\quad \check p^\alpha \in [\hat p^\alpha,p^\alpha]$$
which implies that $\check p^\alpha= \hat p^\alpha$ and so \eqref{eq::z33}.\bigskip

\noindent {\bf Case 2: $H^\alpha(p^\alpha)\le F(p)$.}
A proof similar to the one of Case 1 gives again \eqref{eq::z33}.\\

\noindent {\bf Step 2: Gluing all together.}
In step 1, the function $w$ has been constructed on each branch separately, but satisfies $-w_t(t,0)=\lambda=F(p)$, i.e.
$$w(t,0)=-\lambda t$$
Therefore $w$ is continuous at $x=0$, as defined on each branch $J_\alpha$ for each $\alpha=1,\dots,N$.
We simply have to check that $w$ satisfies the junction condition at $x=0$, i.e.
$$w_t + F(\partial_1 w(t,0),\dots,\partial_N w(t,0))=0.$$
Using \eqref{eq::z33}, this is equivalent to show that
$$F(p)=F(\hat p)$$
which is true since $F$ is constant on the box $[\underline p,\overline p]$ and $p,\hat p\in [\underline p,\overline p]$.
Therefore $w$ is solution to Riemann problem, and by uniqueness of this solution, we conclude that $w=u$.
This ends the proof of the theorem.

\end{proof}
\appendix
\section{Appendices}
In these appendices, we give some interesting results concerning relaxation, which are not necessary for our paper. Since these results seem important, we only give the statement without giving the proof.

\subsection{Further properties of the relaxation}
We begin by a simple characterization of the relaxation at the characteristic points.
\begin{Proposition}[Characterization of  effective junction conditions)]\label{pro::75}
Let $H$ satisfying \eqref{eq::p00} and  $F_0$, $F$  satisfying \eqref{eq::m11}.
Then $F=\frak R F_0$ is characterized by the following properties
$$\left\{\begin{array}{l}
F= \underline{R} F  = \overline{R} F,\\
\\
F\ge F_0\quad \mbox{on}\quad \overline{\chi} (F) ,\\
\\
F\le F_0\quad \mbox{on}\quad  \underline{\chi} (F),
\end{array}\right.$$
\end{Proposition}

As an application of  Godunov relaxation formula to standard Hamilton-Jacobi equations, 
we can easily prove the following result.
We consider here a standard Hamilton-Jacobi equation for a single Hamiltonian $\tilde H:\R\to \R$
\begin{equation}\label{eq::rep59bis}
\tilde{u}_t + \tilde H(\tilde{u}_x)=0 \quad \mbox{for all}\quad (t,x)\in (0,T)\times \Omega \quad \mbox{with}\quad \Omega=(a,b)\ni 0
\end{equation}

Then we have the following result

\begin{Theorem}[Classical viscosity solutions are solutions at a single point]\label{th::rep60bis}
Assume that $\tilde H$  satisfies condition \eqref{eq::p00}, and let $G_{\tilde H}$ be the Godunov flux associated to $\tilde H$, defined as in \eqref{eq::reg43bis}.\\
\noindent {\bf i) (Subsolutions)} Let $ \tilde{u}: (0,T)\times \Omega \to \R$ be a subsolution to \eqref{eq::rep59bis}.
Then $\tilde{u}$ satisfies for all $t\in (0,T)$
\begin{equation}\label{eq::e77bis}
\tilde{u}_t(t,0) + {G}_{\tilde H}( \tilde{u}_x(t,0^-), \tilde{u}_x(t,0^+)) \le 0
\end{equation}
\noindent {\bf ii) (Supersolutions)} Let $ \tilde{u}: (0,T)\times \Omega\to \R$ be a  supersolution of \eqref{eq::rep59bis}.
Then $\tilde{u}$ satisfies for all $t\in (0,T)$
\begin{equation}\label{eq::e78bis}
\tilde{u}_t(t,0) + {G}_{\tilde H}( \tilde{u}_x(t,0^-), \tilde{u}_x(t,0^+)) \ge 0
\end{equation}
\end{Theorem}

\subsection{Refined gradient estimates on a junction}

In this subsection, we state a result concerning Lipschitz estimates for the solution on the junction.
\begin{Theorem}[Refined gradient estimates on a junction]\label{th::b1}
Assume that $H$ satisfies condition \eqref{eq::p00} and that $F_0$ satisfies \eqref{eq::m11} on the box 
$$Q= \prod_{\alpha=1,\dots,N} [m^\alpha,M^\alpha]$$
Let the initial data $u_0$ be Lipschitz continuous and satisfy the following gradient estimate
\begin{equation}\label{eq::b4}
m^\alpha \le (u_0)_x \le M^\alpha \quad \mbox{a.e. on }J_\alpha^* \quad \mbox{for each}\quad \alpha=1,\dots, N
\end{equation}
with $m^\alpha<M^\alpha$.
If these bounds satisfy the following inequalities
\begin{equation}\label{eq::b2}
\left\{\begin{array}{l}
H^\alpha(M^\alpha)\ge F_0(m^1,\dots,m^{\alpha-1}, M^\alpha,m^{\alpha+1},\dots,m^N),\\
H^\alpha(m^\alpha)\le F_0(M^1,\dots,M^{\alpha-1}, m^\alpha,M^{\alpha+1},\dots,M^N),
\end{array}\right|\quad \mbox{for all}\quad \alpha=1,\dots,N
\end{equation}
then there exists the solution $u$ to \eqref{eq::m10}, \eqref{eq::p21}  is globally Lipschitz continuous, with $u(t,\cdot)$ satisfying moreover
the following gradient estimate for all $t\ge 0$
$$m^\alpha \le u_x(t,\cdot) \le M^\alpha \quad \mbox{a.e. on }J_\alpha^* \quad \mbox{for each}\quad \alpha=1,\dots, N$$
\end{Theorem}

 The result is particularly interesting and the proof follows easily from the construction using an implicit scheme which mimics the PDE.
We do not know how to recover (even heuristically) conditions \eqref{eq::b2} directly from the PDE, without using the scheme.

\subsection{Explicit relaxation: the finite dimensional manifold of limiters}\label{s5}
There is a good case, where the relaxation formula can be constructed explicitly,
and the set of effective junction conditions is parametrized by a finite-dimensional set $\mathcal A$.
To this end, we assume that $H$ satisfies \eqref{eq::p00} and we assume moreover that each Hamiltonian $H^\alpha$ 
has a finite number of minima and maxima. Then we show how we can parametrize explicitly all effective boundary conditions by a finite dimensional manifold $\mathcal A$ of limiters.

For $\alpha=1,\dots,N$, we call $m^\alpha_k$ for  $k=1,\dots n_\alpha$, the increasing sequence of points of minima of ${H}^\alpha$ and  $M^\alpha_k$ for $k=1,\dots, n_\alpha-1$ the increasing sequence of points of maxima of ${H}^\alpha$, such that 
$$m^\alpha_0=M^\alpha_{0}:=-\infty< m^\alpha_{1}<M^\alpha_{1} < m^\alpha_{2}<\dots < M^\alpha_{n_\alpha-1}<m^\alpha_{n_\alpha}< +\infty =: M^\alpha_{n_\alpha}=m^{\alpha}_{n_\alpha+1}.$$
For each index $I=(i_1,\dots,i_N)$, we assume that
$$\max_{\alpha=1,\dots,N} {H}^\alpha(m^\alpha_{i_\alpha}) =:\quad A_{I}^0\le  B_{I}^0 \quad := \min_{\alpha=1,\dots,N} {H}^\alpha(M^\alpha_{i_\alpha})$$
with the convention that ${H}^\alpha(\pm \infty)=+\infty$ for $\alpha=1,\dots,N$. We also denote by $(e_1,\dots,e_N)$ the orthonormal basis of $\R^N$.\\

\noindent {\bf Definition of limiters associated to $F_0$.}\\
We now give the recipe to construct flux limiters $A$ given a general function $F_0:\R^N\to \R$ satisfying \eqref{eq::m11}.
Notice that $m_I< M_I$ (if none of the indices $i^\alpha$ is equal to zero), and then $F_0(M_I)\le F_0(m_I)$.

Given $F_0$, we now define a tensor $A=(A_I)_{I} \in \R^{n_1\times \dots\times n_N}$ where the $A_I$ are the associated limiters given by
$$A_I:=\left\{\begin{array}{ll}
\displaystyle \min(A_{I}^0,  \min_{\alpha=1,\dots,N}A_{I-e_\alpha})&\quad  {\rm if} \quad {F}_0(m_I) < A_{I}^0,\\
\displaystyle \max(B_{I}^0, \max_{\alpha=1,\dots,N}A_{I+e_\alpha})&\quad  {\rm if} \quad {F}_0(M_I) > B_{I}^0,\\
{F}_0(p_I) = {H}^\alpha(p^\alpha_{I}) \quad \mbox{for all}\quad \alpha=1,\dots,N &\quad  \mbox{if}\quad A_{I}^0 \le  {F_0}(m_I)  \quad \mbox{and}\quad {F_0}(M_I) \le  B_{I}^0,\\
\end{array}\right.$$
where $p_I=(p^1_I,\dots,p^N_I)$ is the unique solution satisfying $p^\alpha_I \in [m^\alpha_I, M^\alpha_I]$.
Here we use the convention that the tensor $A$ is assumed to be extended for all 
$I\in \prod_{\alpha=1,\dots,N} \left\{0,\dots,n_\alpha\right\}$ by
$A_{I}=+\infty$  if there exists some $\alpha$ such that  $i_\alpha=0$, 
and  for all $I\in \prod_{\alpha=1,\dots,N} \left\{1,\dots,n_\alpha+1\right\}$ by
$A_{I}=-\infty$ if there exists some $\alpha$ such that $i_\alpha=n_\alpha+1$. They satisfy in particular $A_I\ge A_{I'}$ if $I\le I'$.\\

\noindent {\bf Parametrization of the level sets of $F:=\frak R F_0$.}\\
We define  $p^-_I(\mu)=(p^{1-}_{i_1}(\mu),\dots,p^{N-}_{i_N}(\mu))$ with
$$p^{\alpha-}_{k}(\mu):=\left\{\begin{array}{ll}
 M_{k-1}^\alpha &\quad \mbox{if}\quad \mu > H^\alpha(M_{k-1}^\alpha),\\
 m_{k}^\alpha &\quad \mbox{if}\quad  \mu< H^\alpha(m_{k}^\alpha),\\
 p^{\alpha-}_{k} \in \left[M_{k-1}^\alpha, m_{k}^\alpha \right]  \quad \mbox{with}\quad 
H^\alpha(p^{\alpha-}_{k})=\mu  &\quad \mbox{if}\quad \mu \in \left[H^\alpha(m_{k}^\alpha), H^\alpha(M_{k-1}^\alpha)\right].\\
\end{array}\right.$$
We use the convention that $p^{\alpha-}_{1}(+\infty)=-\infty$
and $p^{\alpha-}_{n_\alpha+1}(+\infty)=+\infty$.
For $g\in \left\{-1,0\right\}^N$, we define the quadrant
$$Q_g:= \prod_{\alpha=1,\dots,N} \Delta^\alpha\quad \mbox{with}\quad  \Delta^\alpha:=\left\{\begin{array}{ll}
\left(-\infty,0\right]&\quad \mbox{if}\quad g^\alpha=-1,\\
\left[0,+\infty\right)&\quad \mbox{if}\quad g^\alpha=0.
\end{array}\right.$$

\begin{Theorem}[Explicit relaxation for Hamiltonians $H^\alpha$ with $n_\alpha$ minima]\label{th::25}
Under the previous assumptions, we associate to each function $F_0$,  a tensor of limiters $A=(A_I)_I\in \R^{n_1\times \dots\times n_N}$ which is generically parametrized by a manifold $\mathcal A$ of dimension at most equal to $n_1\times \dots\times n_N$. 

Then the relaxed junction function $F:=\frak R F_0$ only depends on the tensor $A$ and is given as follows.
For all $p\in \prod_{\alpha=1,\dots,N} \left[M^\alpha_{i_\alpha-1}, M^\alpha_{i_\alpha}\right]$, we have
$$F(p)=\inf \left\{\mu\in \R,\quad p\in p^-_I(\mu)+  S_I(\mu)\right\} \quad \mbox{with}\quad S_I(\mu):=\displaystyle \bigcup_{g\in \left\{-1,0\right\}^N,\ A_{I+g}\le \mu} Q_g.$$
\end{Theorem}

\paragraph{\textbf{Acknowledgement.}}
This research was partially funded by l'Agence Nationale de la Recherche (ANR), project ANR-22-CE40-0010 COSS. 
For the purpose of open access, the authors have applied a CC-BY public copyright licence to any Author Accepted Manuscript (AAM) version arising from this submission. The authors thanks C. Imbert for fruitful discussions that strongly  motivated this study.

The second  author also thanks J. Dolbeault, C. Imbert, T. Leli\`{e}vre and G. Stoltz for providing him good working conditions.

\bibliographystyle{siam}
\begin{thebibliography}{ABC}
\bibitem{ACCT} 
{Y. Achdou, F. Camilli, A. Cutr\`{\i}, and N. Tchou},
{\it  Hamilton-Jacobi equations constrained on networks}. Nonlinear Differential Equations and Applications NoDEA, 20(3) (2013), 413-445.

\bibitem{BC}  
\textsc{G. Barles, E. Chasseigne},
{\it Some Comparison Results for First-Order
Hamilton-Jacobi Equations and Second-Order Fully
Nonlinear Parabolic Equations with Ventcell
Boundary Conditions}. Preprint

\bibitem{BC-book} 
{G. Barles, and E. Chasseigne}, 
{\sc On modern approaches of Hamilton-Jacobi equations and control problems with discontinuities. A guide to theory, applications, and some open problems}, volume 104 of Prog. Nonlinear Differ. Equ. Appl. Cham: Birkhaüser, 2024.

\bibitem{FIM1}
\textsc{N. Forcadel, C. Imbert, R. Monneau},
{\it Non-convex coercive Hamilton-Jacobi equations: Guerand's relaxation revisited}, Pure and Applied Analysis 6 (4) (2024), 1055-1089.

\bibitem{FIM2}
\textsc{N. Forcadel, C. Imbert, R. Monneau},
{\it Coercive Hamilton-Jacobi equations in domains:
the twin blow-ups method}, C.R. Math. 362 (2024), 829-839.

\bibitem{FIM3}
\textsc{N. Forcadel, C. Imbert, R. Monneau},
{\it The twin blow-up method for Hamilton-Jacobi equations in higher dimension}, to 	appear in ESAIM:COCV.

\bibitem{guerand}
\textsc{J. Guerand},
{\it  Flux-limited solutions and state constraints for quasi-convex Hamilton-Jacobi equations in multidimensional domains.} Nonlinear Anal. 162 (2017), 162-177.

\bibitem{IM1}
\textsc{C. Imbert, R. Monneau},
{\it Flux-limited solutions for quasi-convex {Hamilton}-{Jacobi} equations on networks},
Ann. Sci. {\'E}c. Norm. Sup{\'e}r. (4) 50 (2) (2017), 357-448.

\bibitem{IMZ}
{C. Imbert, R. Monneau and H. Zidani},
{\it  A Hamilton-Jacobi approach to junction problems and application to traffic flows}. 
ESAIM: Control, Optimisation and Calculus of Variations 19 (2013), 129-166.

\bibitem{LS1} 
{P.-L. Lions and P. Souganidis},
{\it Viscosity solutions for junctions: well posedness and stability}. 
Rendiconti Lincei-matematica e applicazioni, 27(4) (2016), 535-545.

{\bibitem{LS2}
{P.-L. Lions, P. Souganidis},
{\it Well-posedness for multi-dimensional junction problems with Kirchoff-type conditions},
Rend. Lincei Mat. Appl. 28 (2017), 807-816.}

{\bibitem{M}
{R. Monneau},
{\it Structure of Riemann solvers on networks (preliminary version)}, HAL preprint hal-04764513 , version 2 (2025).

\end{document}